\documentclass[reqno,11pt]{amsart}
\usepackage{amsmath, latexsym, amsfonts, amssymb, amsthm, amscd}
\usepackage{epsfig}
\usepackage{graphics,psfrag}

\usepackage[pdftex]{hyperref}

\numberwithin{equation}{section}

\newtheorem{theorem}{Theorem}[section]
\newtheorem{lemma}[theorem]{Lemma}
\newtheorem{proposition}[theorem]{Proposition}
\newtheorem{corollary}[theorem]{Corollary}

\renewenvironment{itemize}{%
\begin{list}{$\bullet$\rule[-0.25em]{0pt}{1.5em}}%
 	{\setlength\leftmargin{1.8em}}%
     \setlength{\labelsep}{0.45em}
	}%
{\end{list}}

\newcommand{\cE}{{\ensuremath{\mathcal E}} }

\newcommand{\cI}{{\ensuremath{\mathcal I}} }

\newcommand{\cN}{{\ensuremath{\mathcal N}} }

\newcommand{\cP}{{\ensuremath{\mathcal P}} }

\newcommand{\ga}{\alpha}
\newcommand{\gb}{\beta}

\newcommand{\gd}{\delta}

\newcommand{\gep}{\varepsilon}       

\newcommand{\gl}{\lambda}

\newcommand{\gs}{\sigma}

\renewcommand{\tilde}{\widetilde}          
\DeclareMathSymbol{\leqslant}{\mathalpha}{AMSa}{"36} 
\DeclareMathSymbol{\geqslant}{\mathalpha}{AMSa}{"3E} 
\DeclareMathSymbol{\eset}{\mathalpha}{AMSb}{"3F}     
\newcommand{\dd}{\text{\rm d}}             

\DeclareMathOperator*{\union}{\bigcup}       
\DeclareMathOperator*{\inter}{\bigcap}       

\title{Depinning of a polymer in a multi-interface medium}

\author{Francesco Caravenna}

\address{Dipartimento di Matematica Pura e Applicata, Universit\`a
degli Studi di Padova, via \mbox{Trieste} 63, 35121 Padova, Italy}

\email{francesco.caravenna\@@math.unipd.it}

\author{Nicolas P\'etr\'elis}

\address{Eurandom, P.O. Box 513, 5600 MB Eindhoven, The Netherlands.}

\email{petrelis\@@eurandom.tue.nl}

\keywords{Polymer Model, Pinning Model, Random Walk, Renewal Theory,
Localization/Delocalization Transition}

\subjclass[2000]{60K35, 60F05, 82B41}

\date{\today}

\newcommand{\R}{\mathbb{R}}
\newcommand{\Z}{\mathbb{Z}}
\newcommand{\N}{\mathbb{N}}

\DeclareMathOperator{\var}{Var}

\def\bP{\ensuremath{\bs{\mathrm{P}}}}
\def\bE{\ensuremath{\bs{\mathrm{E}}}}

\newcommand{\ind}{\bs{1}}

\def\bs{\boldsymbol}

\begin{document}

\begin{abstract}
In this paper we consider a model which
describes a polymer chain interacting with an infinity of
equi-spaced linear interfaces. The distance between two consecutive
interfaces is denoted by $T = T_N$ and is
allowed to grow with the size~$N$ of the polymer.
When the polymer receives a positive reward
for touching the interfaces,
its asymptotic behavior has been derived in \cite{cf:CP},
showing that a transition occurs when $T_N \approx \log N$.
In the present paper, we deal with the so--called
{\sl depinning case}, i.e., the polymer
is repelled rather than attracted by the interfaces.
Using techniques from renewal theory,
we determine the scaling behavior of the model for
large~$N$ as a function of $\{T_N\}_{N}$, showing that
two transitions occur, when $T_N \approx  N^{1/3}$ and
when $T_N \approx \sqrt{N}$ respectively.
\end{abstract}

\maketitle


\medskip

\section{Introduction and main results}


\smallskip
\subsection{The model}

We consider a $(1+1)$-dimensional model of a polymer
depinned at an infinity of equi-spaced horizontal
interfaces. The possible configurations of the polymer
are modeled by the trajectories of the simple random walk $(i,S_i)_{i\geq 0}$, where
$S_0=0$ and
$(S_i-S_{i-1})_{i \geq 1}$ is an i.i.d. sequence of symmetric Bernouilli trials taking values $1$ and $-1$, that is
$P(S_i-S_{i-1}=+1) = P(S_i-S_{i-1}=-1) = \frac 12$.
The polymer receives an energetic penalty $\delta<0$ each times it touches
one of the horizontal interfaces located at heights $\{k T\colon k\in\mathbb{Z}\}$, where $T\in 2\mathbb{N}$
(we assume that $T$ is even for notational convenience).
More precisely, the polymer interacts
with the interfaces through the following Hamiltonian:
\begin{equation}\label{eq:H}
H^T_{N,\delta}(S) \;:= \; \delta\, \sum_{i=1}^{N} \ind_{\{S_i \,\in\, T\Z\}}
\;=\;\delta\, \sum_{k\in\mathbb{Z}}\sum_{i=1}^{N} \ind_{\{S_i\,=\,k\, T\}},
\end{equation}
where $N \in \N$ is the number of monomers constituting the polymer.
We then introduce the
corresponding polymer measure $\bP^T_{N,\gd}$
(see Figure~\ref{fig:1} for a graphical description) by
\begin{equation}\label{eq:model}
\frac{\dd \bP^T_{N,\delta}}{\dd P}(S) \;:=\;
\frac{\exp\big(H^T_{N,\delta}(S)\big)}{Z^T_{N,\delta}},
\end{equation}
where the normalizing constant
$Z^T_{N,\delta} = E[\exp(H^T_{N,\delta}(S))]$ is called the {\sl partition function}.

\begin{figure}[t]
\includegraphics[width=.84\textwidth]{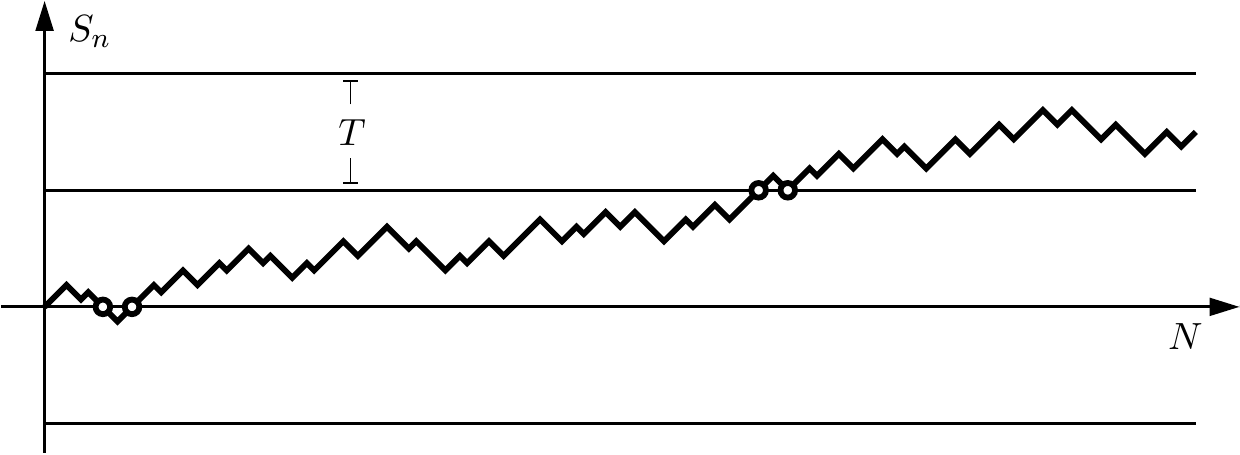}
\caption{A typical path of $\{S_n\}_{0 \le n \le N}$
the polymer measure $\bP^T_{N,\delta}$, for $N=158$
and $T=16$.
The circles indicate the points where the polymer
touches the interfaces, that are penalized by $\gd < 0$ each.}
\label{fig:1}
\end{figure}

We are interested in the case where the interface spacing
$T=\{T_N\}_{N\geq 1}$ is allowed to vary with the size $N$ of the polymer.
More precisely, we aim at understanding whether and how the
asymptotic behavior of the polymer is modified
by the interplay between the energetic penalty $\delta$ and the growth rate of $T_N$
as $N \to \infty$.
In the {\sl attractive case} $\delta>0$, when the polymer is rewarded
rather than penalized to touch an interface, this question
was answered in depth in a previous paper \cite{cf:CP},
to which we also refer for a detailed discussion on the motivation
of the model and for an overview on the literature (see also \S\ref{sec:slit} below).
In the present paper we extend the analysis to the {\sl repulsive case}
$\gd < 0$, showing that 
the behavior of the model is sensibly different from the attractive case.

\smallskip

For the reader's convenience, and in order to get some intuition
on our model, we recall briefly the result obtained in \cite{cf:CP} for $\gd > 0$.
We first set some notation:
given a positive sequence $\{a_N\}_N$, we write
$S_N \asymp a_N$ to indicate that, on the one hand, $S_N / a_N$ is tight
(for every $\gep > 0$ there exists $M > 0$ such that
$\bP_{N,\gd}^{T_N} \big( |S_N/a_N| > M \big) \le \gep$ for
large $N$) and, on the other hand, that for some
$\rho \in (0,1)$ and $\eta > 0$ we have $\bP_{N,\gd}^{T_N}
\big( |S_N/a_N| > \eta \big) \ge \rho$ for large $N$.
This notation catches the rate of asymptotic growth of $S_N$
somehow precisely: if $S_N \asymp a_N$ and
$S_N \asymp b_N$, for some $\gep > 0$
we must have $\gep a_N \le b_N \le \gep^{-1} a_N$, for large $N$.

Theorem~2 in \cite{cf:CP} can be read as follows:
for every $\gd >0$ there exists $c_\gd > 0$ such that
\begin{equation} \label{eq:asdelta>0}
	S_N \text{ under } \bP_{N,\gd}^{T_N}
	\; \asymp \; \begin{cases}
	\sqrt{N} \, e^{-\frac{c_\delta }{2} T_N}\, T_N  & \text{if }
	T_N - \frac{1}{c_\gd} \log N \to -\infty\\
	T_N & \text{if } T_N - \frac{1}{c_\gd} \log N = O(1)\\
	1 & \text{if } T_N - \frac{1}{c_\gd} \log N \to +\infty
	\end{cases}\,.
\end{equation}
Let us give an heuristic explanation for these scalings.
For fixed $T \in 2\N$, the process $\{S_n\}_{0 \le n \le N}$
under $\bP_{N,\gd}^{T}$ behaves approximately like a time-homogeneous
Markov process (for a precise statement in this direction
see \S\ref{sec:renewal}). A quantity of basic interest is the
first time $\hat \tau := \inf\{n > 0:\, |S_n| = T \}$ at which
the polymer visits a neighboring interface. It turns
out that for $\gd > 0$ the typical size of $\hat \tau$
is of order $\approx e^{c_\gd T}$, so that until epoch $N$
the polymer will make approximately $N/e^{c_\gd T}$ changes
of interface. 

Assuming that these arguments can be applied
also when $T = T_N$ varies with $N$, it follows that
the process $\{S_n\}_{0 \le n \le N}$ jumps from an interface
to a neighboring one a number of times which is
approximately $u_N := N/e^{c_\gd T_N}$.
By symmetry, the probability of jumping to the
neighboring upper interface
is the same as the probability of jumping to the lower one,
hence the last visited interface will be approximately
the square root of the number of jumps. Therefore,
when $u_N \to \infty$, one expects that
$S_N$ will be typically of order $T_N \cdot \sqrt{u_N}$,
which matches perfectly with the first line of \eqref{eq:asdelta>0}.
On the other hand, when $u_N \to 0$ the polymer will never visit
any interface different from the one located at zero and, because
of the attractive reward $\gd > 0$, $S_N$ will be typically at
finite distance from this interface, in agreement with the
third line of \eqref{eq:asdelta>0}. Finally, when $u_N$ is
bounded, the polymer visits a finite number of different interfaces
and therefore $S_N$ will be of the same order as $T_N$,
as the second line of \eqref{eq:asdelta>0} shows.


\smallskip
\subsection{The main results}

Also in the repulsive case $\gd < 0$ one can perform an
analogous heuristic analysis. The big difference with respect
to the attractive case is the following: under $\bP_{N,\gd}^T$,
the time $\hat \tau$ the polymer needs to jump
from an interface to a neighboring one turns out to be typically of order $T^3$
(see Section~\ref{sec:preliminary}).
Assuming that these considerations can be applied
also to the case when $T = T_N$ varies with~$N$,
we conclude that, under $\bP_{N,\gd}^{T_N}$, the total number of jumps from
an interface to the neighboring one
should be of order $v_N := N/T_N^3$.
One can therefore conjecture that if $v_N \to +\infty$
the typical size of $S_N$ should be of order $T_N \cdot \sqrt{v_N} = \sqrt{N/T_N}$,
while if $v_N$ remains bounded one should have $S_N \asymp T_N$.

In the case $v_N \to 0$, the polymer will never exit the interval
$(-T_N, +T_N)$. However, guessing the right scaling in this case
requires some care: in fact, due to the
repulsive penalty $\gd < 0$, the polymer will {\sl not} remain close
to the interface located at zero, as it were for $\gd > 0$,
but it will rather spread in the
interval $(-T_N, +T_N)$. We are therefore led to distinguish two cases:
if $T_N = O(\sqrt{N})$
then $S_N$ should be of order $T_N$, while if $T_N \gg \sqrt{N}$
we should have $S_N \asymp \sqrt{N}$ (of course we write $a_N \ll b_N$
iff $a_N / b_N \to 0$ and $a_N \gg b_N$ iff $a_N / b_N \to +\infty$).
We can sum up these considerations in the following formula:
\begin{equation} \label{eq:asdelta<0}
	S_N \; \asymp \; \begin{cases}
	\sqrt{N/T_N} & \text{if }\ T_N \ll N^{1/3}\\
	T_N & \text{if }\ (const.) N^{1/3} \le T_N \le (const.) \sqrt{N}\\
	\sqrt{N} & \text{if }\ T_N \gg \sqrt{N}
	\end{cases}\,.
\end{equation}
It turns out that these conjectures are indeed correct:
the following theorem makes this precise, together with
some details on the scaling laws.

\medskip

\begin{theorem} \label{th:main}
Let $\delta<0$ and $\{T_N\}_{N\in\N} \in (2\N)^{\N}$
be such that $T_N \to \infty$ as $N\to\infty$.
\begin{enumerate}
\item \label{part:1}
\rule{0pt}{1.3em}If $\,T_N \ll N^{1/3}$, then
$S_N \asymp \sqrt{N/T_N}$. More precisely,
there exist two constants $0 < c_1 < c_2 < \infty$ such that
for all $a,b \in \R$ with $a < b$ we have for $N$ large enough
\begin{equation} \label{eq:infinite}
	c_1 \, P\big[ a < Z \le b \big] \;\le\;
	\bP_{N,\delta}^{T_N} \left( a <
	\frac{S_N}{C_\delta \, \textstyle\sqrt{\frac{N}{T_N}}} \le b \right)
    \;\le\; c_2 \, P\big[ a < Z \le b \big] \,,
\end{equation}
where $C_\gd := \pi / \sqrt{e^{-\delta}-1}$
is an explicit positive constant and $Z \sim \cN(0,1)$.

\item \label{part:2}
\rule{0pt}{1.3em}If $\,T_N \sim (const.) N^{1/3}$, then
$S_N \asymp T_N$. More precisely,
for every $\gep > 0$ small enough there exist constants $M,\eta>0$
such that $\forall N\in\N$
\begin{equation}\label{eq:crit}
	\bP_{N,\delta}^{T_N} \big(|S_N| \le M \, T_N\big)
	\;\ge\; 1-\gep \,, \qquad
	\bP_{N,\delta}^{T_N} \big(|S_N| \ge \eta \, T_N\big)
	\;\ge\; 1-\gep \,.
\end{equation}

\item \label{part:3}
\rule{0pt}{1.3em}If $\,N^{1/3} \ll T_N \le (const.)\sqrt{N}$, then
$S_N \asymp T_N$. More precisely,
for every $\gep > 0$ small enough there exist
constants $L,\eta > 0$ such that $\forall N\in \N$
\begin{equation}\label{eq:supercrit1}
	\bP_{N,\delta}^{T_N} \big(
	0 < |S_n| < T_N \,, \ \forall n \in \{ L, N\} \big) \;\ge\; 1- \gep\,,
	\qquad \bP_{N,\delta}^{T_N} \big( |S_N| \ge  \eta \, T_N \big) \;\ge\; 1-\gep \,.
\end{equation}

\item \label{part:4}
\rule{0pt}{1.3em}If $\,T_N \gg \sqrt{N}$, then
$S_N \asymp \sqrt N$. More precisely,
for every $\gep > 0$ small enough there exist
constants $L,M,\eta > 0$ such that $\forall N\in \N$
\begin{equation}\label{eq:supercrit2}
	\bP_{N,\delta}^{T_N} \big(
	0 < |S_n| < M \sqrt{N} \,, \ \forall n \in \{ L, N\} \big) \;\ge\; 1-\gep\,,
	\qquad \bP_{N,\delta}^{T_N} \big( |S_N| \ge  \eta  \sqrt{N} \big)
	\;\ge\; 1-\gep \,.
\end{equation}

\end{enumerate}
\end{theorem}

\medskip

To have a more intuitive view on the scaling behaviors
in \eqref{eq:asdelta<0}, let us consider the concrete example
$T_N \sim (const.) N^a$: in this case we have
\begin{equation} \label{eq:scalings}
	S_N \; \asymp \; \begin{cases}
	N^{(1-a)/2} & \text{if } 0 \le a \le \frac 13\\
	N^a & \text{if } \frac 13 \le a \le \frac 12\\
	N^{1/2} & \text{if } a \ge \frac 12
	\end{cases}\,.
\end{equation}
As the speed of growth of $T_N$ increases, in a first
time (until $a=\frac 13$) the scaling of $S_N$
decreases, reaching a minimum $N^{1/3}$, after which it
increases to reattain the initial value $N^{1/2}$, for $a \ge \frac 12$.

We have thus shown that the asymptotic behavior of our model displays two
transitions, at $T_N \approx \sqrt{N}$ and at $T_N \approx N^{1/3}$.
While the first one is somewhat natural,
in view of the diffusive behavior of the simple random walk,
the transition happening at $T_N \approx N^{1/3}$ is
certainly more surprising and somehow unexpected.

\smallskip

Let us make some further comments on Theorem~\ref{th:main}.

\begin{itemize}
\item About regime (\ref{part:1}), that is when $T_N \ll N^{1/3}$,
we actually conjecture that equation \eqref{eq:infinite} can be strengthened to
a full convergence in distribution:
$S_N/(C_\gd \sqrt {N/T_N}) \Longrightarrow \cN(0,1)$.
The reason for the slightly weaker result that we present is that
we miss precise renewal theory estimates for a basic renewal process,
that we define in \S\ref{sec:renewal}. As a matter of fact,
using the techniques in \cite{cf:Ney} one can refine
our proof and show that the full convergence in
distribution holds true in the restricted regime $T_N \ll N^{1/6}$,
but we omit the details for conciseness
(see however the discussion following Proposition~\ref{th:bound_renewal}).

\item Equation \eqref{eq:infinite}
implies that the sequence $\{S_N/(C_\gd \sqrt {N/T_N})\}_N$
is {\sl tight}, and that the limit law of any converging subsequence 
is absolutely continuous w.r.t. the Lebesgue
measure on $\mathbb{R}$. Moreover, the density of this limit law is bounded above and below
by a multiple of the standard normal density.

\item The case when $T_N \to T \in \R$ as $N\to\infty$ has not been included
in Theorem~\ref{th:main} for the sake of simplicity. However a straightforward
adaptation of our proof shows that in this case equation \eqref{eq:infinite}
still holds true, with $C_\gd$ replaced by a different
($T$-dependent) constant $\widehat C_\gd(T)$.

\item We stress that in regimes (\ref{part:3}) and (\ref{part:4})
the polymer really touches the interface at zero a finite
number of times, after which it does not touch any other interface.

\end{itemize}


\smallskip
\subsection{A link with a polymer in a slit}
\label{sec:slit}

It turns out that our model $\bP^T_{N,\gd}$ is closely related to a model
which has received quite some attention in the recent physical literature,
the so-called {\sl polymer confined between two attractive walls}
\cite{cf:Brak,cf:Martin,cf:Owczarek}
(also known as polymer in a slit). This is a model for the
steric stabilization and sensitized flocculation of colloidal dispersions
induced by polymers,
which can be simply described as follows:
given $N,T \in 2\N$, take the first $N$ steps of the simple
random walk constrained not to exit the interval $\{0,T\}$,
and give each trajectory a reward/penalization $\gamma \in \R$
each time it touches $0$ or $T$ (one can also consider two different
rewards/penalties $\gamma_0$ and $\gamma_T$, but we will stick to the
case $\gamma_0 = \gamma_T = \gamma$). We are thus considering
the probability measure $Q_{N,\gamma}^T$ defined by
\begin{equation} \label{eq:phys}
	\frac{\dd Q_{N,\gamma}^T}{\dd P_N^{c,T}}(S)
	\;\propto\; \exp\left( \gamma \sum_{i=1}^N
	\ind_{\{S_i = 0 \text{ or } S_i = T\}} \right)
	\,,
\end{equation}
where $P_N^{c,T}(\,\cdot\,) := P(\,\cdot\,|\, 0 \le S_i \le T \text{ for all } 0 \le i \le N)$
is the law of the simple random walk {\sl constrained} to stay between
the two walls located at $0$ and $T$.

\begin{figure}[t]
\includegraphics[width=.84\textwidth]{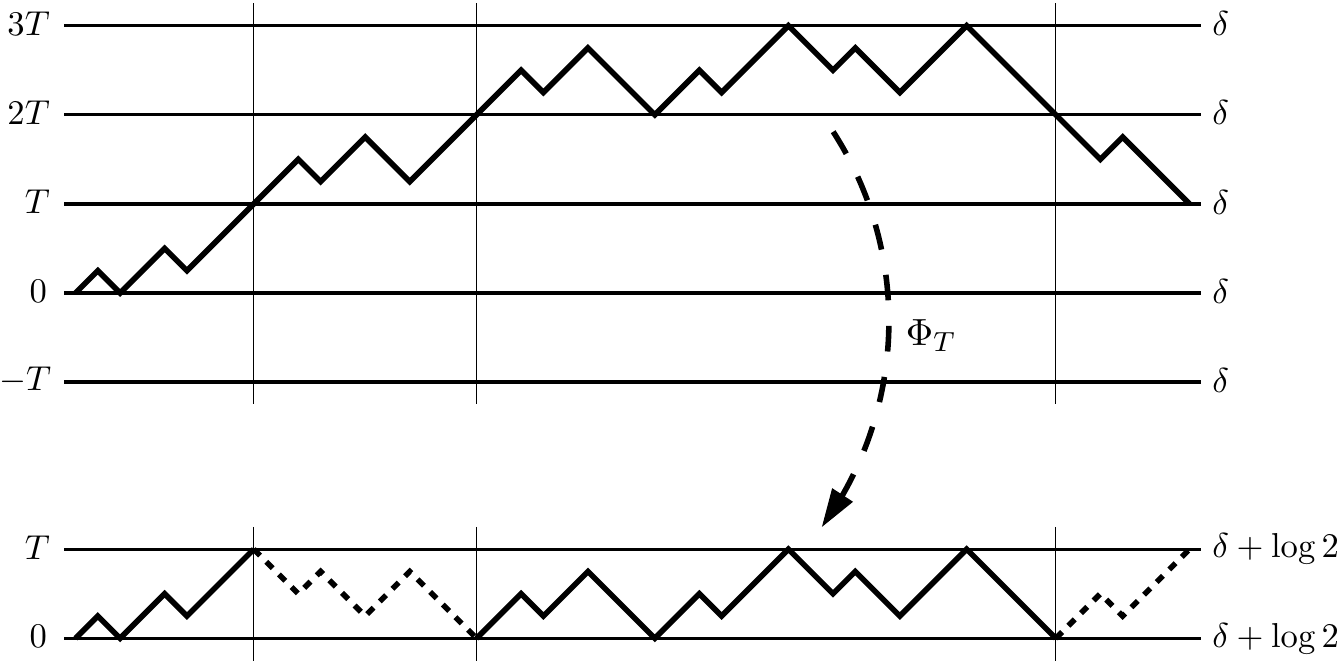}
\smallskip
\caption{A polymer trajectory in a multi-interface medium transformed,
after reflection on the interfaces $0$ and $T$,
in a trajectory of polymer in a slit. The dotted
lines correspond to the parts of trajectory that appear
upside-down after the reflection.}
\label{fig:multuni1}
\end{figure}

Consider now the simple random walk
{\sl reflected} on both walls $0$ and $T$, which may be defined as
$\{\Phi_T(S_n)\}_{n\ge 0}$, where $(\{S_n\}_{n\ge 0}, P)$ is the
ordinary simple random walk and
\begin{equation*}
	\Phi_T(x) \;:=\; \min \big\{\,[x]_{2T}, 2T - [x]_{2T} \big\}\,,
	\qquad \text{with} \qquad
	[x]_{2T} \;:=\: 2T\, \Big(\frac{x}{2T} - \Big\lfloor \frac{x}{2T} \Big\rfloor\Big)\,,
\end{equation*}
that is, $[x]_{2T}$
denotes the equivalence class of $x$ modulo $2T$ (see Figure~\ref{fig:multuni1}
for a graphical description). We denote by $P_N^{r,T}$ the law of
the first $N$ steps of $\{\Phi_T(S_n)\}_{n\geq 0}$. Of
course, $P_N^{r,T}$ is different from $P_N^{c,T}$: the latter is the uniform measure on
the simple random walk paths $\{S_n\}_{0 \le n \le N}$ that stay in
$\{0,T\}$, while under the former each such path has a probability which
is proportional to $2^{\cN_N}$, where
$\cN_N = \sum_{i=1}^N \ind_{\{S_i = 0 \text{ or } S_i = T\}}$ is the
number of times the path has touched the walls. In other terms, we have
\begin{equation} \label{eq:phys2}
	\frac{\dd P_N^{c,T}}{\dd P_N^{r,T}} (S) \;\propto\;
	\exp \left( -(\log 2) \sum_{i=1}^N \ind_{\{S_i = 0 \text{ or } S_i = T\}} \right) \,.
\end{equation}

If we consider the reflection under $\Phi_T$ of our model, that is
the process $\{\Phi_T(S_n)\}_{0 \le n \le N}$ under $\bP^T_{N,\gd}$,
whose law will be simply denoted by $\Phi_T(\bP^T_{N,\gd})$, then it comes
\begin{equation}\label{eq:phyphy}
\frac{\dd \Phi_T(\bP^T_{N,\gd})}{\dd P_N^{r,T}} (S) \;\propto\;
	\exp \left( \delta \sum_{i=1}^N \ind_{\{S_i = 0 \text{ or } S_i = T\}} \right) \,.
\end{equation}
At this stage, a look at equations \eqref{eq:phys}, \eqref{eq:phys2} and \eqref{eq:phyphy} points
out the link with our model:
we have the basic identity $Q^T_{N,\gd + \log 2} = \Phi_T(\bP^T_{N, \gd})$, for
all $\gd \in \R$ and $T,N \in 2\N$. In words, the polymer confined
between two attractive walls is just the reflection of our
model through $\Phi_T$, up to a shift of the pinning intensity by $\log 2$.
This allows a direct translation of all our results in this new framework.

\smallskip

Let us describe in detail a particular issue, namely,
the study of the model $Q^T_{N,\gamma}$
when $T = T_N$ is allowed to vary with $N$
(this is interesting, e.g., in order to interpolate between the two extreme
cases when one of the two quantities $T$ and $N$ tends to $\infty$ before the other).
This problem is considered in \cite{cf:Owczarek}, where the authors obtain
some asymptotic expressions for the partition function
$Z_{n,w}(a,b)$ of a polymer in a slit, in the case of two different rewards/penalties
(we are following their notation, in which
$n=N$, $w = T$, $a = \exp(\gamma_0)$ and $b = \exp(\gamma_T)$)
and with the boundary condition $S_N = 0$.
Focusing on the case $a = b = \exp(\gamma)$,
we mention in particular equations (6.4)--(6.6) in~\cite{cf:Owczarek},
which for $a < 2$ read as
\begin{equation} \label{eq:them1}
	Z_{n,w}(a,a) \;\approx\; \frac{(const.)}{n^{3/2}} \, f_{\text{phase}}
	\left( \frac{\sqrt n}{w} \right)\,,
\end{equation}
where we have neglected a combinatorial factor $2^n$ (which just comes from
a different choice of notation), and where
the function $f_{\text{phase}}(x)$ is such that
\begin{equation} \label{eq:them2}
	f_{\text{phase}}(x) \;\to\; 1 \ \text{ as } \ x \to 0\,, \qquad
	f_{\text{phase}}(x) \;\approx\; x^3 \, e^{-\pi^2 x^2 / 2}
	\ \text{ as } \ x \to \infty \,.
\end{equation}
The regime $a < 2$ corresponds to $\gamma < \log 2$, hence, in view of
the correspondence $\gd = \gamma - \log 2$ described above, we are exactly in
the regime $\gd < 0$ for our model $\bP^T_{N, \gd}$.
We recall \eqref{eq:model} and, with the help of equation
\eqref{eq:overandover}, we can express the partition function
with boundary condition $S_N \in (2T)\Z$ as
\begin{equation*}
	Z^{T,\, \{S_N\in (2T)\Z\}}_{N,\gd} \;\sim\; O(1)\, 
	Z^{T,\,\{S_N\in T\Z\}}_{N,\gd} \;\sim\; O(1) \, e^{\phi(\gd,T) N} \,
	\cP_{\gd, T}(N \in \tau) \,,
\end{equation*}
where, with some abuse of notation, we denote by $O(1)$ a quantity which
stays bounded away from $0$ and $\infty$ as $N \to\infty$. In this formula,
$\phi(\gd, T)$ is the {\sl free energy} of our model and
$(\{\tau_n\}_{n \in \Z^+}, \cP_{\gd, T})$ is a basic renewal process,
introduced respectively in \S\ref{sec:free_energy} and \S\ref{sec:renewal} below.
In the case when $T = T_N \to \infty$, we can use the asymptotic development
\eqref{eq:phineg} for $\phi(\gd, T)$, which, combined with the bounds
in \eqref{eq:bound_renewal}, gives as $N,T \to \infty$
\begin{equation*}
	Z^{T,\, \{S_N\in (2T)\Z\}}_{N,\gd} \,=\, \frac{O(1)}{N^{3/2}} \;
	\max \bigg\{ 1, \bigg( \frac{\sqrt N}{T} \bigg)^3 \bigg\} \;
	\exp \left(- \frac{\pi^2}{2} \frac{N}{T^2} + \frac{2 \pi^2}{e^{-\gd}-1}
	\frac{N}{T^3}
	+ o\left(\frac{N}{T^3}\right) \right).
\end{equation*}
Since $Z^{T,\,\{S_N\in (2T)\Z\}}_{N,\gd} = Z_{n,w}(a,a)$, we can rewrite
this relation using the notation of~\cite{cf:Owczarek}:
\begin{equation*}
	Z_{n,w}(a,a) \;\approx\; \frac{(const.)}{n^{3/2}} \, f_{\text{phase}}
	\left( \frac{\sqrt n}{w} \right) \, g \bigg( \frac{n^{1/3}}{w} \bigg) \,,
	\quad \ \text{where} \ \ g(x) \;\approx\; e^{\frac{2\pi^2}{e^{-\gd}-1} x}
	\ \ \text{as} \ x \to \infty\,.
\end{equation*}

We have therefore obtained a refinement of
equations \eqref{eq:them1}, \eqref{eq:them2}.
This is linked to the fact that we have gone beyond the first order in
the asymptotic development of the free energy $\phi(\gd, T)$, making
an additional term of the order $N/T_N^3$ appear.
We stress that this new term gives a non-negligible (in fact, exponentially diverging!)
contribution as soon as $T_N \ll N^{1/3}$ ($w \ll n^{1/3}$ in the notation
of~\cite{cf:Owczarek}).
This corresponds to the fact that, by Theorem~\ref{th:main},
the trajectories that touch the walls a number of times of the order $N/T_N^3$
are actually dominating the partition function when $T_N \ll N^{1/3}$.
Of course, a higher order development of the free energy
(cf. Appendix~\ref{sec:fe_estimates}) may lead to
further correction terms.

\smallskip
\subsection{Outline of the paper}
Proving Theorem \ref{th:main} requires to settle some technical tools,
partially taken from~\cite{cf:CP},
that we present in Section \ref{sec:preliminary}.
More precisely,
in \S\ref{sec:free_energy} we introduce the free energy $\phi(\delta,T)$
of the polymer and we describe its asymptotic behavior
as $T\to \infty$ (for fixed $\delta<0$).
In \S\ref{sec:renewal} we enlighten a basic correspondence between the polymer
constrained to hit one of the interfaces at its right extremity and an
explicit renewal process.
In \S\ref{sec:asymp} we investigate further this renewal process,
providing estimates on the renewal function, which are of
crucial importance for the proof of Theorem~\ref{th:main}.
Sections \ref{sec:parti}, \ref{sec:partii},
\ref{sec:partiii} and \ref{sec:partiv} are dedicated respectively
to the proof of parts (\ref{part:1}), (\ref{part:2}), (\ref{part:3})
and (\ref{part:4}) of Theorem~\ref{th:main}. Finally,
some technical results are proven in the appendices.

We stress that the value of $\gd < 0$ is kept fixed throughout
the paper, so that the generic constants appearing in the proofs
may be $\gd$-dependent.


\medskip

\section{A renewal theory viewpoint}

\label{sec:preliminary}

In this section we recall some features of our model, including
a basic renewal theory representation, originally proven in \cite{cf:CP},
and we derive some new estimates.


\smallskip
\subsection{The free energy}

\label{sec:free_energy}

Considering for a moment our model when $T_N \equiv T \in 2\N$ is fixed,
i.e., it does not vary with $N$, we define the {\sl free energy}
$\phi(\delta, T)$ as the rate of exponential growth
of the partition function $Z_{N,\gd}^T$ as $N\to\infty$:
\begin{equation}\label{eq:fe}
    \phi( \gd, T) \;:=\; \lim_{N\to\infty} \, \frac 1N
    \, \log Z^{T}_{N,\gd} \;=\;
    \lim_{N\to\infty} \, \frac 1N
    \, \log \, E \left( e^{H_{N,\gd}^T} \right) \,.
\end{equation}
Generally speaking, the reason for looking at this function is
that the values of $\gd$ (if any) at which $\gd \mapsto \phi(\gd, T)$ is
not analytic correspond physically to the occurrence of a
{\sl phase transition} in the system.
As a matter of fact, in our case $\gd \mapsto \phi(\gd, T)$ is analytic
on the whole real line, for every $T \in 2\N$.
Nevertheless, the free energy $\phi(\gd, T)$ turns out to be a very useful
tool to obtain a path description of our model, even when
$T = T_N$ varies with $N$, as we explain in detail
in \S\ref{sec:renewal}. For this reason, we now recall some basic facts
on $\phi(\gd, T)$, that were proven in \cite{cf:CP}, and we derive its
asymptotic behavior as $T\to\infty$.

We introduce $\tau_1^T := \inf\{ n>0:\, S_n\in\{-T, 0, +T\} \}$,
that is the first epoch at which the polymer visits an
interface, and we denote by $Q_T(\gl):= E \big( e^{-\gl \tau_1^T} \big)$
its Laplace transform under the law of the simple random walk.
We point out that $Q_T(\gl)$ is finite and
analytic on the interval $(\gl_0^T, \infty)$,
where $\gl_0^T 
< 0$,
and $Q_T(\gl) \to +\infty$ as $\gl \downarrow \gl_0^T$ (as a matter of fact,
one can give a closed explicit expression for $Q_T(\gl)$, cf.
equations (A.4) and (A.5) in \cite{cf:CP}). A basic fact is that $Q_T(\cdot)$
is sharply linked to the free energy: more precisely, we have
\begin{equation}\label{eq:energie}
\phi(\gd, T) = (Q_T)^{-1}(e^{-\gd}),
\end{equation}
for every $\gd \in \R$
(see Theorem~1 in~\cite{cf:CP}). From this, it is easy to obtain
an asymptotic expansion of $\phi(\gd, T)$ as $T \to \infty$, for fixed
$\gd < 0$, which reads as
\begin{equation} \label{eq:phineg}
    \phi(\gd,T) \;=\; - \frac{\pi^2}{2T^2} \bigg( 1 -
    \frac{4}{e^{-\gd} - 1}\,\frac 1T + o\bigg( \frac 1T \bigg) \bigg)\,,
\end{equation}
as we prove in Appendix~\ref{sec:fe_estimates}.


\smallskip

\subsection{A renewal theory interpretation}
\label{sec:renewal}

We now recall a basic renewal theory description of our model,
that was proven in \S2.2 of~\cite{cf:CP}.
We have already introduced the first epoch $\tau_1^T$ at which
the polymer visits an interface. Let us extend this definition:
for $T\in 2\mathbb{N}\cup \{\infty\}$, we set $\tau^T_{0}=0$ and for $j\in \N$
\begin{equation}\label{jump}
\tau^T_{j} \;:=\; \inf\big\{n > \tau^T_{j-1}: \ S_n\in T \Z \big\}
\qquad \text{and} \qquad
\varepsilon^T_{j} \;:=\; \tfrac{S_{\tau^T_{j}}-S_{\tau^T_{j-1}}}{T}\,,
\end{equation}
where for $T = \infty$ we agree that $T\Z = \{0\}$. Plainly,
$\tau^T_j$ is the $j^{\text{th}}$ epoch at which $S$
visits an interface and $\varepsilon^T_j$
tells whether the $j^{\text{th}}$ visited interface
is the same as the $(j-1)^{\text{th}}$ ($\varepsilon^T_j=0$),
or the one above
($\varepsilon^T_j=1$) or below ($\varepsilon^T_j=-1$).
We denote by $q_T^j(n)$ the joint law of $(\tau^T_1, \gep^T_1)$
under the law of the simple random walk:
\begin{equation}\label{eq:defQ}
    q^j_{T}(n) \;:=\; P\big( \tau^T_1=n\,,\, \varepsilon^T_1=j \big)\,.
\end{equation}
Of course, by symmetry we have that $q^1_T(n) = q^{-1}_T(n)$ for every $n$ and $T$.
We also set
\begin{equation} \label{eq:deftau}
	q_T(n) \;:=\; P\big( \tau_1^T = n \big) \;=\; q_T^0(n) \,+\,
	2 \, q_T^1(n) \,.
\end{equation}

Next we introduce a Markov chain $(\{(\tau_j, \gep_j)\}_{j \ge 0}, \cP_{\gd, T})$
taking values in $(\N\cup\{0\}) \times \{-1, 0, 1\}$,
defined in the following way: $\tau_0 := \gep_0 := 0$ and under $\cP_{\gd, T}$
the sequence of vectors $\{(\tau_j - \tau_{j-1}, \gep_j)\}_{j \ge 1}$ is i.i.d.
with marginal distribution
\begin{equation} \label{eq:defPdeltaT}
	\cP_{\gd, T}(\tau_1 = n,\, \gep_1 = j)
	\;:=\; e^\gd \, q^j_T(n) \, e^{-\phi(\gd, T) \, n} \,.
\end{equation}
The fact that the r.h.s. of this equation indeed defines a probability
law follows from \eqref{eq:energie}, which implies that $Q(\phi(\gd,T)) = E(e^{-\phi(\gd,T) \tau_1^T}) = e^{-\gd}$.
Notice that the process $\{\tau_j\}_{j \ge 0}$ alone under $\cP_{\gd, T}$
is a (undelayed) {\sl renewal process}, i.e. $\tau_0 = 0$ and
the variables $\{\tau_j - \tau_{j-1}\}_{j \ge 1}$ are i.i.d., with step law
\begin{equation} \label{eq:taudelta}
	\cP_{\gd, T}(\tau_1 = n ) \;=\; e^{\gd} \, q_T(n) \, e^{-\phi(\gd,T) n}
	\;=\; e^{\gd} \, P(\tau_1^T = n) \, e^{-\phi(\gd,T) n} \,.
\end{equation}

Let us now make the link between the law $\cP_{\gd, T}$ and our model
$\bP_{N,\gd}^T$. We introduce two variables that count how many epochs
have taken place before $N$, in the processes $\tau^T$ and $\tau$ respectively:
\begin{equation} \label{eq:L}
	L_{N,T} \;:=\; \sup \big\{ n \ge 0:\ \tau_n^T \le N \big\}\,,
	\qquad
	L_{N} \;:=\; \sup \big\{ n \ge 0:\ \tau_n \le N \big\} \,.
\end{equation}
We then have the following crucial result (cf. equation (2.13) in \cite{cf:CP}):
for all $N,T \in 2\N$
and for all $k\in\N$, $\{t_i\}_{1 \le i \le k} \in \N^k$,
$\{\gs_i\}_{1 \le i \le k} \in \{-1, 0, +1\}^k$ we have
\begin{align} \label{eq:crucial}
\begin{split}
	& \bP_{N,\delta}^{T} \Big( L_{N,T} = k,\ (\tau_i^T , \gep_i^T) = (t_i,\gs_i),\,
	1 \le i \le k  \,\Big|\, N \in \tau^T \Big)\\
	& \qquad \quad \;=\; \cP_{\delta,T} \Big( L_{N} = k,\ (\tau_i , \gep_i)
	= (t_i,\gs_i),\,
	1 \le i \le k \,\Big|\, N \in \tau \Big)\,,
\end{split}
\end{align}
where $\{N \in \tau\} := \union_{k=0}^\infty \{\tau_k = N\}$
and analogously for $\{N \in \tau^T\}$.
In words, the process $\{(\tau_j^T, \gep_j^T)\}_{j}$ under
$\bP_{N,\gd}^T(\,\cdot\,| N\in\tau^T)$ is distributed like
the Markov chain $\{(\tau_j, \gep_j)\}_j$
under $\cP_{\gd, T}(\,\cdot\,|N\in\tau)$.
It is precisely this link with renewal theory that makes our model
amenable to precise estimates.
Note that the law $\cP_{\gd, T}$ carries no explicit
dependence on~$N$.
Another basic relation we are going to use repeatedly is the following one:
\begin{equation} \label{eq:overandover}
	E \left[ e^{H^{T}_{k, \gd}(S)} \, \ind_{\{k \in \tau^T\}} \right]
	\;=\; e^{\phi(\gd, T) k} \, \cP_{\gd, T} \big(k \in \tau \big) \,,
\end{equation}
which is valid for all $k, T \in 2\N$ (cf. equation (2.11) in \cite{cf:CP}).


\smallskip

\subsection{Some asymptotic estimates}
\label{sec:asymp}

We now derive some estimates that will be used throughout the paper.
We start from the asymptotic behavior of $P(\tau_1^T = n)$
as $n\to\infty$. Let us set
\begin{equation} \label{eq:gT}
	g(T) \;:=\; -\log \cos \left( \frac{\pi}{T} \right)
	\;=\; \frac{\pi^2}{2 T^2} + O\left( \frac{1}{T^4} \right)\,,
	\qquad (T \to \infty)\,.
\end{equation}
We then have the following

\medskip
\begin{lemma}\label{th:ineg2}
There exist positive constants $T_0, c_1, c_2, c_3, c_4$
such that when $T > T_0$ the following relations hold
for every $n \in 2\N$:
\begin{gather} \label{eq:boundq}
	\frac{c_1}{\min\{T^3, n^{3/2}\}}\, e^{-g(T) n} \;\le\;
	P(\tau_1^T = n) \;\le\; \frac{c_2}{\min\{T^3, n^{3/2}\}}
	\, e^{-g(T) n} \,,\\
	\label{eq:boundqbis}
	\frac{c_3}{\min\{T, \sqrt{n}\}}\, e^{-g(T) n} \;\le\;
	P(\tau_1^{T} > n) \;\le\; \frac{c_4}{\min\{T, \sqrt{n}\}}
	\, e^{-g(T) n} \,.
\end{gather}
\end{lemma}
\medskip

The proof of Lemma~\ref{th:ineg2} is somewhat technical and is
deferred to Appendix~\ref{sec:lemmaineg2}.
Next we turn to the study of the
renewal process $\big( \{\tau_n\}_{n \ge 0}, \cP_{\gd, T} \big)$.
It turns out that the law of $\tau_1$ under $\cP_{\gd, T}$
is essentially split into two components:
the first one at $O(1)$, with mass $e^\gd$, and the second one at $O(T^3)$,
with mass $1-e^\gd$ (although we do not fully prove these results, it is useful
to keep them in mind). We start with the following estimates
on $\cP_{\gd, T}(\tau_1 = n )$, which follow quite easily from Lemma~\ref{th:ineg2}.

\medskip
\begin{lemma}\label{th:good}
There exist positive constants $T_0, c_1, c_2, c_3, c_4$
such that when $T > T_0$ the following relations hold
for every $m, n \in 2\N \cup \{+\infty\}$ with $m < n$:
\begin{align} \label{eq:boundren}
	\frac{c_1}{\min\{T^3, k^{3/2}\}}\, e^{-(g(T) + \phi(\gd,T)) k} & \;\le\;
	\cP_{\gd, T} (\tau_1 = k) \;\le\; \frac{c_2}{\min\{T^3, k^{3/2}\}}
	\, e^{-(g(T) + \phi(\gd,T)) k} \\
	\label{eq:boundrenbislb}
	\cP_{\gd, T}(m \le \tau_1 < n) & \;\ge\;
	c_3 \, \left( e^{-(g(T) + \phi(\gd,T)) m} -
	e^{-(g(T) + \phi(\gd,T)) n} \right) \\
	\label{eq:boundrenbisub}
	\cP_{\gd, T}(\tau_1 \ge m)
	& \;\le\; c_4 \, e^{-(g(T) + \phi(\gd,T)) m} \,.
\end{align}
\end{lemma}
\medskip

\begin{proof}
Equation \eqref{eq:boundren} is an immediate consequence of equations
\eqref{eq:taudelta} and \eqref{eq:boundq}. To prove \eqref{eq:boundrenbislb},
we sum the lower bound in \eqref{eq:boundren} over $k \in 2\N$,
observing that by \eqref{eq:phineg} and \eqref{eq:gT},
for every fixed $\gd < 0$, we have as $T\to\infty$
\begin{equation} \label{eq:gplusphi}
	g(T) \,+\, \phi(\gd, T) \;=\; \frac{4 \pi^2}{2(e^{-\gd}-1)} \,
	\frac{1}{T^3} \, \big( 1 + o(1) \big) \,.
\end{equation}
To get \eqref{eq:boundrenbisub}, we sum the upper bound in \eqref{eq:boundren}
over $k \in 2\N$ and we are done.
\end{proof}
\medskip

Notice that equation \eqref{eq:boundren}, together with
\eqref{eq:gplusphi}, shows indeed that the law of $\tau_1$
has a component at $O(T^3)$, which is approximately geometrically distributed.
Other important asymptotic relations are the following ones:
\begin{align}\label{eq:asET}
    \cE_{\gd, T}(\tau_1) \;&=\; \frac{e^\gd (e^{-\gd}-1)^2}{2 \pi^2} \, T^3
    \;+\; o(T^3)\,,\\
    \label{eq:asET2}
    \cE_{\gd, T}(\tau_1^2) \;&=\; \frac{e^\gd (e^{-\gd}-1)^3}{2 \pi^4} \, T^6
    \;+\; o(T^6)\,,
\end{align}
which are proven in Appendix~\ref{sec:further_estimates}. We stress
that these relations, together with equation \eqref{eq:asQ1bis},
imply that, under $\cP_{\gd, T}$, the time $\hat \tau$ needed to hop
from an interface to a neighboring one is of order $T^3$, and this is
precisely the reason
why the asymptotic behavior of our model has a transition at $T_N \approx N^{1/3}$,
as discussed in the introduction.
Finally, we state an estimate on the renewal function
$\cP_{\gd, T}(n \in \tau)$, which is proven in Appendix~\ref{sec:bound_renewal}.

\medskip
\begin{proposition}\label{th:bound_renewal}
There exist positive constants $T_0, c_1, c_2$
such that for $T > T_0$ and for all $n \in 2\N$ we have
\begin{gather} \label{eq:bound_renewal}
	\frac{c_1}{\min\{n^{3/2}, T^3\}} \;\le\;
	\cP_{\delta,T} (n \in \tau) \;\le\;
	\frac{c_2}{\min\{n^{3/2}, T^3\}}\,.
\end{gather}
\end{proposition}
\medskip

Note that the large $n$ behavior of \eqref{eq:bound_renewal}
is consistent with the classical renewal
theorem, because $1/\cE_{\gd,T}(\tau_1) \approx T^{-3}$, by \eqref{eq:asET}.
One could hope to refine this estimate,
e.g., proving that for $n \gg T^3$ one has
$\cP_{\delta,T} (n \in \tau) = (1+o(1))/\cE_{\gd,T}(\tau_1)$:
this would allow strengthening part~(\ref{part:1}) of Theorem~\ref{th:main}
to a full convergence in distribution
$S_N/(C_\gd \sqrt {N/T_N}) \Longrightarrow \cN(0,1)$.
It is actually possible
to do this for $n \gg T^6$, using the ideas and techniques
of~\cite{cf:Ney}, thus strengthening Theorem~\ref{th:main}
in the restricted regime $T_N \ll N^{1/6}$ (we omit the details).


\medskip
\section{Proof of Theorem~\ref{th:main}: part (\ref{part:1})}

\label{sec:parti}

We are in the regime when $N/T_N^3\to \infty$ as $N\to \infty$.
The scheme of this proof is actually very
similar to the one of the proof of part (i) of Theorem 2 in~\cite{cf:CP}.
However, more technical difficulties
arise in this context, essentially because, in the depinning case ($\delta<0$), the density of contact between the polymer
and the interfaces vanishes as $N\to \infty$,
whereas it is strictly positive in the
pinning case ($\delta>0$). For this reason,
it is necessary to display this proof in detail.

Throughout the proof we set
$v_\delta=(1-e^\delta)/2$ and $k_N=\lfloor N/\cE_{\delta,T_N}(\tau_1)\rfloor$.
Recalling \eqref{jump} and \eqref{eq:L}, we set $Y_0^{T_N}=0$ and
$Y_i^{T_N}=\gep_1^{T_N}+\dots+\gep_i^{T_N}$ for $i\in\{1,\dots,L_{N,T_N}\}$.
Plainly, we can write
\begin{equation}\label{eq:simpli}
	S_N \;=\; Y^{T_N}_{L_{N,T_N}} \cdot T_N
	\,+\, s_N \,, \qquad \text{with} \quad  |s_N|\,<\,T_N \,.
\end{equation}
In view of equation \eqref{eq:asET}, this relation shows
that to prove \eqref{eq:infinite} we can equivalently
replace $S_N/(C_\delta \sqrt{N/T_N})$ with
$Y_{L_{N,T_N}}^{T_N}/\sqrt{v_\delta k_N}$.


\smallskip
\subsection{Step 1}

\label{sec:s1}

Recall \eqref{eq:defPdeltaT} and set $Y_n=\gep_1+\dots+\gep_n$ for all $n\geq 1$.
The first step consists in proving that for all $a<b$ in $\overline{\R}$
\begin{equation}\label{step1}
    \lim_{N\to \infty}\; \cP_{\delta,T_N}\Big(a<\frac{Y_{k_{N}}}{\sqrt{v_\delta k_{N}}}
    \leq b\Big) \;=\; P(a<Z\leq b)\,,
\end{equation}
that is, under $\cP_{\delta,T_N}$ and as $N\to \infty$ we have
$Y_{k_{N}}/\sqrt{v_\delta k_{N}} \Longrightarrow Z$, where ``$\Longrightarrow$''
denotes convergence in distribution.

The random variables $(\gep_1,\dots,\gep_N)$, defined under $\cP_{\delta, T_N}$, are symmetric and i.i.d.\;. Moreover, they take
their values in $\{-1,0,1\}$, which together with \eqref{eq:asQ1bis} entails
\begin{equation}\label{eq:nowecant}
	\cE_{\delta,T_N}(|\varepsilon_1|^3) \;=\;
	\cE_{\delta,T_N}((\varepsilon_1)^2) \;\longrightarrow\;
	v_\delta \qquad \text{as} \ N\to \infty.
\end{equation}
Observe that $k_N\to \infty$ as $N\to \infty$ and
$\cE_{\delta,T_N}(\tau_1) = O(T_N^3)$, by \eqref{eq:asET}.
Thus, we can apply the Berry Esse\`en Theorem
that directly proves \eqref{step1} and completes this step.\qed


\smallskip
\subsection{Step 2}

\label{sec:s2}

Henceforth, we fix a sequence of integers $(V_N)_{N\geq 1}$ such that
$T_N^3 \ll V_N \ll N$.
In this step we prove that, for all $a<b \in \overline{\R}$, the
following convergence occurs, uniformly in $u\in\{0,\dots, 2V_N\}$:
\begin{equation}\label{step2}
    \lim_{N\to \infty} \; \cP_{\delta,T_N}
    \Bigg(a<\frac{Y_{L_{N-u}}}{
    \sqrt{v_\delta k_N}}\leq b \Bigg) \;=\; P(a<Z\leq b)\,.
\end{equation}

To obtain \eqref{step2}, it is sufficient to prove that,
as $N\to \infty$ and under the law $\cP_{\delta, T_N}$,
\begin{equation}\label{eq:imp}
U_N:=\frac{Y_{k_N}}{\sqrt{v_\delta k_N}} \Longrightarrow Z \quad \quad
\text{and} \quad \quad
 G_N:=\sup_{u\in\{0,\dots,2V_N\}}\bigg|\frac{Y_{L_{N-u}}-Y_{k_N}}{\sqrt{v_\delta k_N}}\bigg|\Longrightarrow 0 \,.
\end{equation}
Step 1 gives directly the first relation in \eqref{eq:imp}.
To deal with the second relation, we must show that
$\cP_{\delta,T_N}(G_N\geq \gep)\to 0$ as $N\to \infty$, for all $\gep>0$.
To this purpose, notice that
$\{G_N\geq \gep\}\subseteq A_\eta^N \cup B_{\eta,\gep}^N$, where for $\eta > 0$
we have set
\begin{align}
A_\eta^N:&=\big\{L_N-k_N\geq \eta k_N\big\}
\cup\big\{L_{N-2V_N}-k_N\leq -\eta k_N\big\}\\
B_{\eta,\gep}^N:&= \Bigg\{\sup
\bigg\{\bigg|\frac{Y_{k_N+i}-Y_{k_N}}{\sqrt{v_\delta k_N}}\bigg|\,,
\ i \in \{-\eta k_N,\dots,\eta k_N\} \bigg\} \geq \gep\Bigg\} \,.
\end{align}
Let us focus on $\cP_{\delta,T_N}(A_\eta^N)$.
Introducing the centered variables
$\tilde{\tau_k} := \tau_k - k \cdot \cE_{\delta,T_N}(\tau_1)$,
for $k \in \N$, by the Chebychev inequality we can write
(assuming that $(1-\eta)k_N \in \N$ for notational convenience)
\begin{align}\label{yeswecan}
 	\nonumber
	\cP_{\delta, T_N} & \big(L_{N-2V_N}-k_N<-\eta k_N\big)
    \;=\; \cP_{\gd, T_N} \big( \tau_{(1-\eta)k_N} > N-2V_N \big) \\
    \nonumber
    &\;=\; \cP_{\gd, T_N} \big( \tilde\tau_{(1-\eta)k_N} > N- 2V_N
    -(1-\eta) k_N \cE_{\gd,T}(\tau_1) \,=\, \eta N- 2V_N \big) \\
    & \;\le\; \frac{(1-\eta) k_N \var_{\gd, T_N}(\tau_1)}{(\eta N - 2V_N)^2} \;\le\;
    \frac{N\; \var_{\delta,T_N}(\tau_1)}{(\eta N-2V_N)^2\, \cE_{\delta,T_N}(\tau_1)}\,.
\end{align}
With the help of the estimates in \eqref{eq:asET}, \eqref{eq:asET2},
we can assert that
$\var_{\delta,T_N}(\tau_1)/\cE_{\delta,T_N}(\tau_1) = O(T_N^3)$.
Since $N \gg V_N$ and $N \gg T_N^3$, the r.h.s. of~\eqref{yeswecan}
vanishes as $N \to \infty$. With a similar technique, we prove that
$\cP_{\delta, T_N} \big(L_N-k_N>\eta k_N\big)\to 0$
as well, and consequently $\cP_{\delta,T_N}(A_\eta^N)\to 0$ as $N\to \infty$.

At this stage it remains to show that, for every fixed $\gep>0$,
the quantity  $\cP_{\gd, T_N} \big(B_{\eta,\gep}^N \big)$
vanishes as $\eta\to 0$, {\sl uniformly in $N$}. This holds true
because $\{Y_n\}_n$ under $\cP_{\gd, T_N}$ is a symmetric random walk,
and therefore $\{(Y_{k_N + j} - Y_{k_N})^2\}_{j \ge 0}$
is a submartingale (and the same with $j \mapsto  -j$).
Thus, the maximal inequality yields
\begin{equation}\label{eq:soutcha}
    \cP_{\gd, T_N} \big( B_{\eta,\gep}^N \big) \;\le\; \frac{2}{\gep} \,
    \frac{\cE_{\gd, T_N} \big( (Y_{k_N + \eta k_N} - Y_{k_N})^2 \big)}
    {v_\delta k_N} \;\le\;
    \frac{2\, \eta\, \cE_{\gd, T_N}(\gep_1^2)}{\gep v_\delta}
    \;\le\; \frac{2 \, \eta}{\gep \, v_\gd} \,.
\end{equation}
We can therefore assert that the r.h.s in \eqref{eq:soutcha} tends
to $0$ as $\eta\to 0$, uniformly in $N$. This completes the step.\qed



\smallskip
\subsection{Step 3}

\label{sec:s3}

Recall that $k_N=\lfloor N/\cE_{\delta,T_N}(\tau_1)\rfloor$.
In this step we assume for simplicity that $N\in2\N$,
and we aim at switching from the free measure $\cP_{\delta,T_N}$ to
$\cP_{\delta,T_N}\big(\cdot\,\big|\, N\in\tau \big)$.
More precisely, we want to prove that there exist two constants $0<c_1< c_2<\infty$
such that for all $a<b \in\overline{\R}$ there exists $N_0 > 0$
such that for $N\ge N_0$ and for all $u \in \{0,\dots,V_N\} \cap 2\N$
\begin{equation}\label{step3}
  	c_1 \, P(a<Z\leq b) \;\le\; \cP_{\delta,T_N}
    \bigg( a < \frac{Y_{L_{N-u}}}{
    \sqrt{v_\delta k_N}} \leq b \,\bigg|\, N-u \in\tau \bigg)
    \;\le\; c_2 \, P(a<Z\leq b) \,.
\end{equation}

A first observation is that we can safely replace
${L_{N-u}}$ with ${L_{N-u-T_N^3}}$ in \eqref{step3}.
To prove this, since $k_N \to \infty$, the following bound is sufficient:
for every $N, M \in 2\N$
\begin{equation} \label{eq:tec_toprove}
	\sup_{u \in \{0, \ldots, V_N\} \cap 2\N} \;
	\cP_{\delta,T_N} \Big( \big| Y_{L_{N-u}} - Y_{L_{N-u-T_N^3}} \big| \ge M
	\,\Big|\, N-u \in \tau \Big)
    \;\le\; \frac{(const.)}{M} \,.
\end{equation}
Note that the l.h.s. is bounded above by
$\cP_{\delta,T_N} \big( \# \big\{ \tau \cap[N-u-T_N^3, N-u) \big\} \ge M
\,\big|\, N-u \in \tau \big)$. By time-inversion and the renewal property
we then rewrite this as
\begin{equation}\label{eq:barabao}
\begin{split}
	& \cP_{\delta,T_N} \big( \# \big\{ \tau \cap(0, T_N^3] \big\} \ge M
    \,\big|\, N-u \in \tau \big) \;=\;
    \cP_{\delta,T_N} \big( \tau_M \le T_N^3 \,\big|\, N-u \in \tau \big) \\
    & \qquad \;\le\; \sum_{n=1}^{T_N^3} \cP_{\delta,T_N} \big( \tau_M = n \big)
    \cdot \frac{\cP_{\delta,T_N} \big( N-u-n \in \tau \big)}
    {\cP_{\delta,T_N} \big( N-u \in \tau \big)} \,.
\end{split}
\end{equation}
Recalling that $N \gg V_N \gg T_N^3$ and using the estimate
\eqref{eq:bound_renewal}, we see that the ratio in the r.h.s. of
\eqref{eq:barabao} is bounded above by some constant, uniformly
for $0 \le n \le T_N^3$ and $u \in \{0, \ldots, V_N\} \cap 2\N$.
We are therefore left with estimating $\cP_{\delta,T_N} \big( \tau_M \le T_N^3 \big)$.
Recalling the definition $\tilde \tau_k := \tau_k - k \cdot \cE_{\gd, T}(\tau_1)
\sim \tau_k - c k T^3$ as $T \to \infty$, where $c > 0$ by \eqref{eq:asET},
it follows that for large $N\in\N$ we have
\begin{equation*}
\begin{split}
	& \cP_{\gd, T_N}(\tau \cap [0, T_N^3] \ge M) \;=\;
	\cP_{\gd, T_N}(\tau_M \le T_N^3) \\
	& \quad \;\le\; \cP_{\gd, T_N} \bigg( \tilde \tau_M \le
	-\frac c2 \, M \,T_N^3 \bigg) \;\le\;
	\frac{4M \, \var_{\gd,T_N}(\tau_1)}{c^2 \, M^2 \, T_N^6}
	\;\le\; \frac{(const.)}{M}\,,
\end{split}
\end{equation*}
having applied the Chebychev inequality and \eqref{eq:asET2}. This proves
\eqref{eq:tec_toprove}.


Let us come back to \eqref{step3}. By summing over the last point in $\tau$
before $N-u-T_N^3$ (call it $N-u-T_N^3-t$) and the first point in $\tau$
after $N-u-T_N^3$ (call it $N-u-T_N^3+r$), using the Markov property
we obtain
\begin{align} \label{eq:long}
\begin{split}
    & \cP_{\delta,T_N} \Bigg(a< \frac{Y_{L_{N-u-T_N^3}}}{ \sqrt{v_\delta k_N}}
    \leq b \,\Bigg |\, N-u\in\tau \Bigg)\\
    & \ \;=\; \sum_{t=0}^{N-T_N^3-u} \cP_{\gd, T_N}
    \Bigg( a<\frac{Y_{L_{N-u-T_N^3-t}}}{\sqrt{v_\delta k_N}} \leq b \,,\,
    N-u-T_N^3-t \in \tau \Bigg) \cdot \cP_{\gd, T_N}\big( \tau_1 > t \big)
    \cdot \Theta^u_{\gd,N}(t) \,,
\end{split}
\end{align}
where $\Theta^u_{\gd,N}$ is defined by
\begin{equation}\label{theta}
    \Theta^u_{\gd,N}(t) \;:=\; \frac{\sum_{r=1}^{T_N^3}
    \cP_{\gd, T_N}\big(\tau_1 = t+r\big) \cdot \cP_{\gd, T_N}
    \big(T_N^3-r \in \tau \big)}
    {\cP_{\gd, T_N} \big( N-u\in\tau \big) \cdot
   \cP_{\gd, T_N}\big( \tau_1>t \big)}\,.
\end{equation}
Let us set $\cI_N^u:=\{0,\dots,N-u-T_N^3\}$.
Notice that replacing $\Theta^u_{\gd,N}(t)$ by the constant $1$
in the r.h.s. of \eqref{eq:long}, the latter becomes equal to
\begin{equation}\label{eq:intew}
\cP_{\delta,T_N}
\bigg(a<\frac{Y_{L_{N-u-T_N^3}}}{\sqrt{v_\delta k_N}}\leq b \bigg).
\end{equation}
Since $u + T_N^3 \le 2 V_N$ for large $N$ (because
$V_N \gg T_N^3$), equation \eqref{step2} implies that
\eqref{eq:intew} converges as $N\to\infty$ to $P(a < Z \le b)$,
uniformly for $u \in \{0, \ldots, V_N\} \cap 2\N$.
Therefore, equation \eqref{step3} will be proven (completing this step)
once we show that there exists $N_0$ such that
$\Theta^u_{\gd,N}(t)$ is bounded from above and below by two constants
$0<l_1<l_2<\infty$, for $N \ge N_0$ and for
all $u\in\{0,\dots,V_N\}$ and $t\in\cI_N^u$.

Let us set $K_N(n) := \cP_{\gd, T_N}(\tau_1 = n)$ and $u_N(n) :=
\cP_{\gd, T_N}(n \in \tau)$.
The lower bound is obtained by restricting the sum in the numerator of \eqref{theta} to
$r\in\{1,\dots, T_N^3/2\}$. Recalling that $N\gg V_N \gg T_N^3$,
and applying the upper (resp. lower) bound in \eqref{eq:bound_renewal}
to $u_N(N-u)$ (resp. $u_N(T_N^3-r)$), we have that for large $N$,
uniformly in $u\in\{0,\dots,V_N\}$ and $t\in \cI_N^u$,
\begin{equation}\label{eq:etec}
    \Theta^u_{\gd,N}(t) \;\geq \;
    \frac{\sum_{r=1}^{T_N^3/2}
    K_N(t+r)\cdot u_N
    \big(T_N^3-r\big)}
    {u_N\big( N-u\big) \cdot
   \sum_{j=1}^{\infty} K_N(t+j)}
   \;\geq\; \frac{c_1}{c_2} \cdot \frac{\sum_{r=1}^{T_N^3/2}
    K_N(t+r)}
    {\sum_{j=1}^{\infty} K_N(t+j)}\,.
\end{equation}
Then, we use \eqref{eq:boundrenbislb} to bound from below the numerator
in the r.h.s. of \eqref{eq:etec} and we use \eqref{eq:boundrenbisub}
to bound from above its denominator. This allows to write
\begin{equation}\label{eq:etec2}
    \Theta^u_{\gd,N}(t) \;\geq \;
    \frac{c_1\, c_3\, (1-e^{-(g(T_N)+\phi(\delta,T_N))\frac{T_N^3}{2}})}{c_2\, c_4}\,.
\end{equation}
Moreover, \eqref{eq:gplusphi} shows that there exists $m_\delta>0$ such
that $g(T_N)+\phi(\delta,T_N)\sim m_\delta/T_N^3$ as $N\to \infty$, which
proves that the r.h.s. of \eqref{eq:etec2} converges to a constant $c>0$
as $N$ tends to $\infty$. This completes the proof of the lower bound.

The upper bound is obtained by splitting the r.h.s. of \eqref{theta} into
\begin{equation}\label{thetap}
    R_N+D_N\;:=\; \frac{\sum_{r=1}^{T_N^3/2}
    K_N(t+r) \cdot u_N(T_N^3-r)}
    {u_N\big( N-u\big) \cdot
   \sum_{j=1}^{\infty} K_N(t+j)}\,+\,\frac{\sum_{r=1}^{T_N^3/2}
    K_N(t+T_N^3-r) \cdot u_N(r)}
    {u_N\big( N-u\big) \cdot
   \sum_{j=1}^{\infty} K_N(t+j)}.
\end{equation}
The term $R_N$ can be bounded from above by a constant by simply applying
the upper bound in \eqref{eq:bound_renewal}
to $u_N(T_N^3-r)$ for all $r\in\{1,\dots,T_N^3/2\}$ and the lower bound to
$u_N(N-u)$.
To bound $D_N$ from above, we use the upper bound in \eqref{eq:boundren},
which, together with the fact that
$g(T_N)+\phi(\delta,T_N)\sim m_\delta/T_N^3$,
shows that there exists $c>0$ such that for $N$ large enough and
$r\in\{1,\dots,T_N^3/2\}$ we have
\begin{equation}\label{eq:theta2}
	K_N(t+T_N^3-r)\leq \frac{c}{T_N^3}
	\, e^{-(g(T_N) + \phi(\gd,T_N))\, t}.
\end{equation}
Notice also that by \eqref{eq:boundrenbislb} we can assert that
\begin{equation}\label{eq:etac}
\sum_{j=1}^{\infty} K_N(t+j)\geq c_3  e^{-(g(T_N) + \phi(\gd,T_N))\,t}.
\end{equation}
Finally, \eqref{eq:theta2}, \eqref{eq:etac} and the fact that $u_N(N-u)\geq c_1/T_N^3$ for all
$u\in\{0,\dots,V_N\}$ (by \eqref{eq:bound_renewal}) allow to write
\begin{equation}\label{theta3}
    D_N\;\leq\; \frac{c \sum_{r=1}^{T_N^3/2} u_N(r)}
    {c_1 c_3}.
\end{equation}
By applying the upper bound in \eqref{eq:bound_renewal}, we can check easily that $\sum_{r=1}^{T_N^3/2} u_N(r)$ is bounded from above uniformly in $N\geq 1$ by a constant. This completes the proof of the step.\qed


\smallskip
\subsection{Step 4}

\label{sec:s4}

In this step we complete the proof of Theorem~\ref{th:main} (\ref{part:1}),
by proving equation \eqref{eq:infinite}, that we
rewrite for convenience: there exist $0<c_1< c_2<\infty$
such that for all $a<b \in\overline{\R}$ and for large $N\in2\N$
(for simplicity)
\begin{equation}\label{step4}
    c_1 \, P(a<Z\leq b) \;\leq\; \bP_{N,\delta}^{T_N}
    \Bigg( a<\frac{Y^{T_N}_{L_N}}{\sqrt{v_\delta k_N}}\leq b \Bigg)
    \;\leq\; c_2 \, P(a<Z\leq b) \,.
\end{equation}

We recall \eqref{jump} and we start summing over the location $\mu_N := \tau^{T_N}_{L_{N,T_N}}$ of
the last point in $\tau^{T_N}$ before $N$:
\begin{equation}\label{eq:trun}
    \bP_{N,\delta}^{T_N} \Bigg(
    a<\frac{Y_{L_{N,T_N}}^{T_N}}{\sqrt{v_\delta k_N}}\leq b \Bigg)
    \;=\; \sum_{\ell = 0}^N  \; \bP_{N,\delta}^{T_N}
    \Bigg(a<\frac{Y_{L_{N,T_N}}^{T_N}}{\sqrt{v_\delta k_N}}\leq b \,\bigg|\, \mu_N = N-\ell \Bigg)
     \ \cdot \bP_{N,\delta}^{T_N}\big( \mu_N = N-\ell \big)\,.
\end{equation}
Of course, only the terms with $\ell$ even are non-zero.
We want to show that the sum in the r.h.s. of \eqref{eq:trun}
can be restricted to $\ell \in\{0,\dots,V_N\}$.
To that aim, we need to prove that
$\sum_{\ell =V_N }^N \bP_{N,\delta}^{T_N}\big( \mu_N = N-\ell \big)$ tends to $0$
as $N\to \infty$. We start by displaying a lower bound
on the partition function $Z_{N,\gd}^{T_N}$.

\smallskip
\begin{lemma}\label{le:boundzn}
There exists a constant $c>0$ such that for $N$ large enough
\begin{equation}\label{eq:lemest}
	Z_{N, \gd}^{T_N} \;\geq\;
	\frac{c}{T_N} \;  e^{\phi(\gd,T_N) N} \,.
\end{equation}
\end{lemma}

\begin{proof}
Summing over the location of $\mu_N$ and
using the Markov property, together with \eqref{eq:overandover}, we have
\begin{align}\label{eq:egalit}
\nonumber	Z_{N,\gd}^{T_N} & \;=\; E\Big[ e^{H^{T_N}_{N,\gd}(S)} \Big]
	\;=\; \sum_{r=0}^N  E\Big[ e^{H^{T_N}_{N,\gd}(S)} \,
	\ind_{\{\mu_N = r \}} \Big]\\
\nonumber	& \;=\; \sum_{r=0}^N  E\Big[ e^{H^{T_N}_{r,\gd}(S)} \,
	\ind_{\{r \in \tau^{T_N}\}} \Big] \, P(\tau_1^{T_N} > N-r)\\
	& \;=\; \sum_{r=0}^N  e^{\phi(\gd,T_N) r} \,
	\cP_{\gd, T_N} (r \in \tau)\, P(\tau_1^{T_N} > N-r) \,.
\end{align}
From \eqref{eq:egalit} and the lower bounds
in \eqref{eq:boundqbis} and \eqref{eq:bound_renewal}, we obtain for $N$ large enough
\begin{equation}\label{eq:egal2}
	Z_{N,\gd}^{T_N} \;\ge\; (const.)\;
	e^{\phi(\gd,T_N) N} \sum_{r=0}^N\;
	\frac{e^{-[\phi(\delta,T_N)+g(T_N)](N-r)}}
	{\min\{\sqrt{N-r+1}, T_N\}\,\min\{(r+1)^{3/2}, T_N^3 \}} \,.
\end{equation}
At this stage, we recall that
$\phi(\gd,T)+g(T) = m_\gd/T^3 + o(1/T^3)$ as $T\to\infty$,
with $m_\gd > 0$, by \eqref{eq:gplusphi}. Since $T_N^3 \ll N$,
we can restrict the sum in \eqref{eq:egal2}
to $r\in\{N-T_N^3,\dots,N-T_N^2\}$, for large $N$, obtaining
\begin{align}\label{eq:egal3}
	Z_{N,\gd}^{T_N} \;\ge\; (const.)\,\frac{e^{\phi(\gd,T_N) N}}{T_N^4} \,
	\sum_{r=N-T_N^3}^{N-T_N^2}\, e^{-\big( \frac{m_\delta}{T_N^3}+
	o\big( \frac{1}{T_N^3} \big) \big)(N-r)}
	\;\geq\; (const.') \,
	\frac{e^{\phi(\gd,T_N) N}}{T_N} \,,
\end{align}
because the geometric sum gives a contribution of order $T_N^3$.
\end{proof}

We can now bound from above (using the Markov property and \eqref{eq:overandover})
\begin{align}\label{eq:calc}
\nonumber\sum_{l=0}^{N-V_N} \bP_{N,\delta}^{T_N}(\mu_N & =\ell)
\;=\;\sum_{\ell=0}^{N-V_N}
 \frac{\bE \big( \exp\big( H_{\ell, \gd}^{T_N}(S) \big) \ind_{\{\ell \in \tau\}}\big)
 \cdot \bP \big( \tau_1 > N-\ell \big)}{Z_{N,\delta}^{T_N}}\\
\nonumber &\;=\;\sum_{\ell=0}^{N-V_N}
 \frac{\cP_{\delta,T_N}(\ell\in\tau)\  e^{\phi(\delta, T_N)\ell}\ \bP \big( \tau_1 > N-\ell \big)}{Z_{N,\delta}^{T_N}}\\
 & \;\leq\; (const.)\, \sum_{\ell=0}^{N-V_N}
 \frac{T_N}{\min\{(\ell+1)^{3/2}, T_N^3\}}\cdot \frac{e^{-[\phi(\delta, T_N)+g(T_N)] (N-\ell)}}{\min\{\sqrt{N-\ell},T_N\}} \,,
\end{align}
where we have used Lemma \ref{le:boundzn} and the upper bounds in \eqref{eq:boundqbis} and \eqref{eq:bound_renewal}.
For notational convenience we set $d(T_N)=\phi(\delta, T_N)+g(T_N)$.
Then, the estimate \eqref{eq:gplusphi} and the fact that $V_N\gg T_N^3$ imply that
\begin{align}\label{eq:ratic}
\begin{split}
\sum_{\ell=0}^{N-V_N} \bP_{n,\delta}^{T_N}(\mu_N=\ell)&\;\leq\;
(const.)\, e^{-d(T_N) V_N}
\sum_{\ell=0}^{N-V_N}
\frac{e^{-d(T_N) (N-V_N-\ell)}}{\min\{(\ell+1)^{3/2},T_N^3\}}\\
&\;\leq\; (const.')\, e^{-d(T_N) V_N} \Bigg(\sum_{\ell=0}^{\infty}
\frac{1}{(l+1)^{3/2}}+\sum_{\ell=0}^{\infty}
\frac{e^{-d(T_N) (\ell)}}{T_N^3}\Bigg) \,.
\end{split}
\end{align}
Since $d(T_n)\sim m_\delta/T_N^3$, with $m_\gd > 0$,
and $V_N\gg T_N^3$ we obtain that the l.h.s. of \eqref{eq:ratic}
tends to $0$ as $N\to \infty$.

Thus, we can write
\begin{equation}\label{eq:trun2}
\begin{split}
    & \bP_{N,\delta}^{T_N} \Bigg(a<
    \frac{Y_{L_{N,T_N}}^{T_N}}{\sqrt{v_\delta k_N}}\leq b \Bigg) \\
    & \quad \;=\; \sum_{\ell = 0}^{V_N} \bP_{N,\delta}^{T_N}
    \Bigg(a<\frac{Y_{L_{N,T_N}}^{T_N}}{\sqrt{v_\delta k_N}}\leq b \,\Bigg|\,
    \mu_N = N-\ell \Bigg)
    \; \bP_{N,\delta}^{T_N}\big( \mu_N = N-\ell \big)
    \; + \; \gep_N(a,b) \,,
\end{split}
\end{equation}
where $\gep_N(a,b)$ tends to $0$ as $N\to \infty$, uniformly
over $a,b \in\R$. At this stage,
by using the Markov property and \eqref{eq:crucial} we may write
\begin{align*}
	& \bP_{N,\delta}^{T_N}
    \Bigg(a<\frac{Y_{L_{N,T_N}}^{T_N}}{\sqrt{v_\delta k_N}}\leq b
    \,\Bigg|\, \mu_N = N-\ell \Bigg)
    \;=\; \bP_{N,\delta}^{T_N}
    \Bigg(a<\frac{Y_{L_{N-\ell,T_N}}^{T_N}}{\sqrt{v_\delta k_N}}\leq b
    \,\Bigg|\, N-\ell \in \tau^T \Bigg) \\
    & \qquad \;=\; \cP_{\delta,T_N}
    \Bigg(a<\frac{Y_{L_{N-\ell}}}{\sqrt{v_\delta k_N}}\leq b \,\Bigg|\,
    N-\ell\in\tau \Bigg) \,.
\end{align*}
Plugging this into \eqref{eq:trun2},
recalling \eqref{step3} and the fact that
$\sum_{\ell = 0}^{V_N} \bP^{T_N}_{N,\gd}(\mu_N = N-\ell) \to 1$ (by \eqref{eq:ratic}),
it follows that equation \eqref{step4} is proven, and the proof is complete.
\qed


\medskip

\section{Proof of Theorem~\ref{th:main}: part (\ref{part:2})}
\label{sec:partii}

We assume that $T_N \sim (const.) N^{1/3}$ and we start proving
the first relation in \eqref{eq:crit}, that we rewrite as follows:
for every $\gep > 0$ we can find $M>0$ such that for large $N$
\begin{equation*}
	\bP_{N,\gd}^{T_N}  \big( |S_N| > M \cdot T_N \big) \; \le \gep\,.
\end{equation*}
Recalling that $L_N^T$ is the number of times the polymer
has touched an interface up to epoch $N$, see \eqref{eq:L},
we have $|S_N| \le T_N \cdot (L_{N,T_N} + 1)$, hence it suffices to show that
\begin{equation} \label{eq:toproveii}
	\bP_{N,\gd}^{T_N}  \big( L_{N,T_N} > M \big) \; \le \gep\,.
\end{equation}
By using \eqref{eq:crucial} we have
\begin{align*}
	\bP_{N,\gd}^{T_N} \big( & L_{N,T} > M \big)
	\;=\; \frac{1}{Z_{N,\gd}^{T_N}} \, E\Big[ e^{H^{T_N}_{N,\gd}(S)}
	\, \ind_{\{L_{N,T_N} > M\}} \Big] \\
	& \;=\; \frac{1}{Z_{N,\gd}^{T_N}} \,
	\sum_{r=0}^N  E\Big[ e^{H^{T_N}_{r,\gd}(S)} \,
	\ind_{\{L_{r,T_N} > M\}} \,
	\ind_{\{r \in \tau^{T_N}\}} \Big] \, P(\tau_1^{T_N} > N-r) \\
	& \;=\; \frac{1}{Z_{N,\gd}^{T_N}} \,
	\sum_{r=0}^N e^{\phi(\gd, T_N)r} \, \cP_{\gd, T_N} \big( L_{r,T_N} > M ,
	\, r \in \tau^{T_N} \big) \, P(\tau_1^{T_N} > N-r) \,.
\end{align*}
By \eqref{eq:boundqbis} and \eqref{eq:gT} it follows easily that
\begin{equation} \label{eq:plainlb}
	Z^{T_N}_{N,\gd} \;\ge\; P(\tau_1^{T_N} > N) \;\ge\;
	\frac{(const.)}{T_N}\, e^{-\frac{\pi^2}{2 T_N^2} N}
\end{equation}
(note that this bound holds true whenever we have $(const.) N^{1/4} \le T_N \le (const.')
\sqrt{N}$ for large $N$).
Using this lower bound on $Z^{T_N}_{N,\gd}$,
together with the upper bound in \eqref{eq:boundqbis}, the asymptotic developments
in \eqref{eq:gplusphi} and \eqref{eq:gT}, we obtain
\begin{align*}
	\bP_{N,\gd}^{T_N} \big( L_{N,T} > M \big)
	\;\le\; (const.) \, T_N \,
	\sum_{r=0}^N \cP_{\gd, T_N} \big( L_{r,T_N} > M ,
	\, r \in \tau^{T_N} \big) \, \frac{1}{\min\{\sqrt{N-r+1}, T_N\}} \,.
\end{align*}
The contribution of the terms with $r > N-T_N^2$
is bounded with the upper bound \eqref{eq:bound_renewal}:
\begin{equation*}
	T_N \sum_{r=N-T_N^2}^N \frac{1}{T_N^3} \, \frac{1}{\sqrt{N-r+1}}
	\;\le\; \frac{(const.)}{T_N} \;\longrightarrow \;0 \qquad
	(N\to\infty) \,,
\end{equation*}
while for the terms with $r \le N-T_N^2$ we get
\begin{equation*}
	T_N \, \sum_{r=0}^N \cP_{\gd, T_N} \big( L_{r,T_N} > M ,
	\, r \in \tau^{T_N} \big) \, \frac{1}{T_N} \;=\;
	\cE_{\gd, T_N} \big( (L_{N,T_N} - M)
	\ind_{\{L_{N,T_N} > M\}} \big) \,.
\end{equation*}
Finally, we simply observe that $\{L_{N,T_N}=k\} \subseteq
\inter_{i=1}^k\{\tau_i-\tau_{i-1} \le N\}$, hence
\begin{equation*}
	\cP_{\gd, T_N}(L_{N,T_N}=k) \;\le\;
	\big( \cP_{\gd, T_N}(\tau_1 \le N) \big)^k \;\le\; c^k\,,
\end{equation*}
with $0 < c < 1$, as it follows from \eqref{eq:boundrenbislb} and
\eqref{eq:gplusphi} recalling that $N = O(T_N^3)$.
Putting together the preceding estimates, we have
\begin{align*}
	\bP_{N,\gd}^{T_N} \big( L_{N,T_N} > M \big) & \;\le\;
	(const.) \, \cE_{\gd, T_N} \big( (L_{N,T_N} - M)
	\ind_{\{L_{N,T_N} > M\}} \big) \\
	& \;=\; 	(const.) \, \sum_{k=M+1}^\infty (k-M)
	\, \cP_{\gd, T_N}(L_{N,T_N} = k) \\
	& \;\le\; (const.) \, \sum_{k=M+1}^\infty (k-M) \, c^k
	\;\le\; (const.') \, c^M\,,
\end{align*}
and \eqref{eq:toproveii} is proven by choosing $M$ sufficiently large.

\smallskip

Finally, we prove at the same time
the second relations in \eqref{eq:crit} and \eqref{eq:supercrit1}, by showing that
for every $\gep > 0$ there exists $\eta > 0$ such that
for large $N$
\begin{equation} \label{eq:secondpart}
	\bP^{T_N}_{N,\gd} \big( |S_N| \le \eta \, T_N \big) \;\le\; \gep \,,
\end{equation}
whenever $T_N$ satisfies $(const.) N^{1/3} \le T_N \le (const.') \sqrt{N}$
for large $N$.
Letting $P_k$ denoting the law of the simple random walk starting at $k \in \N$
and $\tau_1^\infty$ its first return to zero, it follows by Donsker's invariance
principle that there exists $c>0$ such that
$\inf_{0 \le k \le \eta T_N} P_k (\tau_1^\infty \le \eta^2 T_N^2\,,\
S_i < T_N \,\forall i \le \tau_1^\infty) \ge c$
for large $N$. Therefore we may write
\begin{equation*}
\begin{split}
	& c \,\bP^{T_N}_{N,\gd} \big( |S_N| \le \eta \, T_N \big) \;=\;
	\frac{c}{Z^{T_N}_{N,\gd}} \, \sum_{k= 0}^{\eta T_N}
	E \Big[ e^{H^{T_N}_{N,\gd}(S)} \, \ind_{\{|S_N| = k\}} \Big] \\
	& \;\le\;
	\frac{1}{Z^{T_N}_{N,\gd}} \, \sum_{k= 0}^{\eta T_N}
	\, \sum_{u=0}^{\eta^2 T_N^2} \,
	E \Big[ e^{H^{T_N}_{N,\gd}(S)} \, \ind_{\{|S_N| = k\}} \Big]
	\, P_k(\tau_1^\infty = u\,,\ S_i < T_N \,\forall i \le u) \\
	& \;=\; \frac{1}{Z^{T_N}_{N,\gd}} \, \sum_{k= 0}^{\eta T_N}
	\, \sum_{u=0}^{\eta^2 T_N^2} \,
	E \Big[ e^{H^{T_N}_{N+u,\gd}(S)} \, \ind_{\{|S_N| = k\}}
	\, \ind_{\{|S_{N+i}| < T_N \,\forall i \le u\}} \, \ind_{\{S_{N+u}=0\}} \Big]\,.
\end{split}
\end{equation*}
Performing the sum over $k$, dropping the second indicator function and using
equations \eqref{eq:overandover}, \eqref{eq:bound_renewal} and \eqref{eq:phineg},
we obtain the estimate
\begin{equation*}
\begin{split}
	& \bP^{T_N}_{N,\gd} \big( |S_N| \le \eta \, T_N \big) \;\le\;
	\frac{1}{c\, Z^{T_N}_{N,\gd}} \, \sum_{u=0}^{\eta^2 T_N^2} \,
	E \Big[ e^{H^{T_N}_{N+u,\gd}(S)} \, \ind_{\{N+u \in \tau^{T_N}\}} \Big] \\
	& \;\le\; \frac{1}{c\, Z^{T_N}_{N,\gd}} \, \sum_{u=0}^{\eta^2 T_N^2} \,
	e^{\phi(\gd,T_N) (N+u)} \, \cP_{\gd, T_N}(N+u \in \tau) \;\le\;
	(const.) \, \frac{\eta^2 \,T_N^2}{Z^{T_N}_{N,\gd} \, T_N^3} \,
	e^{-\frac{\pi^2}{2 T_N^2} N}\,.
\end{split}
\end{equation*}
Then \eqref{eq:plainlb} shows that equation \eqref{eq:secondpart}
holds true for $\eta$ small, and we are done.\qed


\medskip

\section{Proof of Theorem~\ref{th:main}: part (\ref{part:3})}
\label{sec:partiii}

We now give the proof of part (\ref{part:3}) of Theorem~\ref{th:main}.
More precisely, we prove the first relation in \eqref{eq:supercrit1},
because the second one has been proven at the end of Section~\ref{sec:partii}
(see \eqref{eq:secondpart} and the following lines).
We recall that we are in the regime when $N^{1/3} \ll T_N \le (const.)\sqrt{N}$,
so that in particular
\begin{equation} \label{eq:ass1}
	C \;:=\; \inf_{N\in\N} \, \frac{N}{T_N^2} \;>\; 0 \,.
\end{equation}
We start stating an immediate corollary of Proposition~\ref{th:bound_renewal}.

\begin{corollary} \label{th:corcor}
For every $\gep > 0$ there exist $T_0 > 0$,
$M_\gep \in 2\N$, $d_\gep > 0$ such that for $T > T_0$
\begin{equation*} \label{eq:renconc}
	\sum_{k=M_\gep}^{d_\gep T^3} \cP_{\delta,T} \big( k\in\tau \big)
	\;\le\; \gep \,.	
\end{equation*}
\end{corollary}

Note that we can restate the first relation in \eqref{eq:supercrit1}
as $\bP^{T_N}_{N,\gd} \big( \tau^{T_N}_{L_{N,T_N}} \le L \big) \ge 1 - \gep$.
Let us define three intermediate quantities, by setting
for $l\in\mathbb{N}$
\begin{align}
	\label{eq:B1}
	B_1(l,N) & \;=\; \bP_{N,\delta}^{T_N}\big(\tau^{T_N}_{L_{N,T_N}}\leq l\big)\,
	Z_{N,\delta}^{T_N} \,,\\
	B_2(l,N) & \;=\; \bP_{N,\delta}^{T_N} \big(l<\tau^{T_N}_{L_{N,T_N}}\leq N-\eta T_N^2\big)
	\, Z_{N,\delta}^{T_N} \,,\\
	B_3(N) & \;=\; \bP_{N,\delta}^{T_N} \big(\tau^{T_N}_{L_{N,T_N}}> N-\eta T_N^2\big)
	\, Z_{N,\delta}^{T_N} \,,
\end{align}
where we fix $\eta := C/2$, so that $\eta T_N^2 \le N/2$.
The first relation in \eqref{eq:supercrit1} will be proven
once we show that for all $\gep>0$, there exists
$l_\gep \in \N$ such that for large $N$ we have
\begin{align}\label{eq:2cond}
	\frac{B_2(l_\gep,N)}{B_1(l_\gep,N)} \;\le\; \gep
	\qquad \text{and} \qquad
	\frac{B_3(N)}{B_1(l_\gep,N)} \;\le\; \gep \,.
\end{align}

We start giving a simple lower bound of $B_1$: since $\{\tau^{T_N}_{L_{N,T_N}}\leq l\}
\supseteq \{\tau_1^{T_N} > N\}$, we have
\begin{align}\label{eq:lbB1}
	B_1(l,N) \;\ge\; E \Big[ e^{H^{T_N}_{N,\gd}(S)} \,
	\ind_{\{\tau_1^{T_N} > N\}} \Big]
	\;=\; P \big( \tau_1^{T_N} > N \big)
	\;\ge\; \frac{(const.)}{T_N} \,e^{-\frac{\pi^2}{2T_N^2} N} \,,
\end{align}
having applied the lower bound in \eqref{eq:boundqbis}.
Next we consider $B_2$.
Summing over the possible values of $\tau_{L_{N,T_N}}^{T_N}$ and using
\eqref{eq:overandover}, we have
\begin{equation}\label{eq:recob}
\begin{split}
	B_2(l,N) & \;=\; \sum_{n=l+1}^{N-\eta T_N^2}
	E \Big[ e^{H^{T_N}_{n,\gd}(S)} \, \ind_{\{n \in \tau^{T_N}\}} \Big]
	\cdot P\big( \tau_1^{T_N} > N - n \big)\\
	& \;=\; \sum_{n=l+1}^{N-\eta T_N^2}
	\mathcal{P}_{\delta, T_N}(n\in\tau) \; e^{\phi(\delta,T_N) n} \;
	P\big(\tau_1^{T_N}>N-n\big) \\
	& \;\le\; \frac{(const.)}{T_N} \, e^{-\frac{\pi^2}{2T_N^2} N} \,
	\left( \sum_{n=l+1}^{N} \mathcal{P}_{\delta, T_N}(n\in\tau) \right) \,,
\end{split}
\end{equation}
where we have applied the upper bound in \eqref{eq:boundqbis} and the equalities \eqref{eq:phineg} and \eqref{eq:gT}
(we also assume that $\eta T_N^2 \in \N$ for simplicity).
Since $N \ll T_N^3$, by Corollary~\ref{th:corcor}
we can fix $l = l_\gep$ depending only on $\gep$
such that $B_2/B_1 \le \gep$ (recall \eqref{eq:lbB1}).
Finally we analyze $B_3(N)$: in analogy with \eqref{eq:recob} we write
\begin{align*}
	B_3(N) & \;\le\; \sum_{n=N - \eta T_N^2 + 1}^{N}
	\mathcal{P}_{\delta, T_N}(n\in\tau) \; e^{\phi(\delta,T_N) n} \;
	P\big(\tau_1^{T_N}>N-n\big) \\
	& \;\le\; e^{-\frac{\pi^2}{2T_N^2} N} \,
	\frac{(const.)}{T_N^3} \, \sum_{n=N - \eta T_N^2 + 1}^{N}
	\frac{(const.')}{\sqrt{N-n+1}} \;\le\;
	(const.'') \, e^{-\frac{\pi^2}{2T_N^2} N} \, \frac{1}{T_N^2} \,,
\end{align*}
where we have applied the upper bounds in \eqref{eq:boundqbis} and
\eqref{eq:bound_renewal} (note that $n \ge (C/2) \, T_N^2$).
Therefore $B_3/B_1 \le \gep$ for $N$ large, and
the first relation in \eqref{eq:supercrit1} is proven.

\medskip

\section{Proof of Theorem~\ref{th:main}: part (\ref{part:4})}
\label{sec:partiv}

We now assume that $T_N \gg \sqrt{N}$, that is
\begin{equation} \label{eq:ass1bis}
	\lim_{N\to\infty} \, \frac{N}{T_N^2} \;=\; 0 \,.
\end{equation}
The proof is analogous to the proof of part~(\ref{part:3}),
given in Section~\ref{sec:partiii}. We set
for $l \in \N$
\begin{align}
	\label{eq:B1new}
	B_1(l,N) & \;=\; \bP_{N,\delta}^{T_N}\big(\tau^{T_N}_{L_{N,T_N}} < l\big)\,
	Z_{N,\delta}^{T_N} \,,\\
	B_2(l,N) & \;=\; \bP_{N,\delta}^{T_N} \big(l\le \tau^{T_N}_{L_{N,T_N}}
	\leq N/2 \big)
	\, Z_{N,\delta}^{T_N} \,,\\
	B_3(N) & \;=\; \bP_{N,\delta}^{T_N} \big( \tau^{T_N}_{L_{N,T_N}}> N/2 \big)
	\, Z_{N,\delta}^{T_N} \,,
\end{align}
and we first show that for every $\gep > 0$ we can choose $l_\gep \in \N$
such that for large $N$
\begin{align} \label{eq:priimo}
	\frac{B_2(l_\gep,N)}{B_1(l_\gep,N)} \;\le\; \gep
	\qquad \text{and} \qquad
	\frac{B_3(N)}{B_1(l_\gep,N)} \;\le\; \gep \,.
\end{align}
We start with a lower bound: since $\{\tau^{T_N}_{L_{N,T_N}}\leq l\}
\supseteq \{\tau_1^{T_N} > N\}$, by \eqref{eq:boundqbis} we have
\begin{align} \label{eq:lbB1new}
	B_1(l,N) \;\ge\; E \Big[ e^{H^{T_N}_{N,\gd}(S)} \,
	\ind_{\{\tau_1^{T_N} > N\}} \Big]
	\;=\; P \big( \tau_1^{T_N} > N \big)
	\;\ge\; \frac{(const.)}{\sqrt{N}} \,.
\end{align}
Next consider $B_2$.
Summing over the possible values of $\tau_{L_{N,T_N}}^{T_N}$ and using
\eqref{eq:overandover}, we have
\begin{equation}
\begin{split}
	B_2(l,N) & \;=\; \sum_{k=l}^{N/2}
	E \Big[ e^{H^{T_N}_{k,\gd}(S)} \, \ind_{\{k \in \tau^{T_N}\}} \Big]
	\cdot P\big( \tau_1^{T_N} > N - k \big)\\
	& \;=\; \sum_{k=l}^{N/2}
	\mathcal{P}_{\delta, T_N}(k\in\tau) \; e^{\phi(\delta,T_N) k} \;
	P\big(\tau_1^{T_N}>N-k\big)
\end{split}
\end{equation}
(we assume that $N/2 \in \N$ for notational convenience).
By the upper bound in \eqref{eq:boundqbis} we have
$P\big(\tau_1^{T_N}>N-k\big) \le (const.')/\sqrt{N-k}$.
Since $\phi(\gd, T_N) \le 0$, we obtain
\begin{equation*}
	B_2(l,N) \;\le\; \frac{(const.)}{\sqrt N} \, \left(
	\sum_{k=l}^{N} \mathcal{P}_{\delta, T_N}(k\in\tau) \right) \,,
\end{equation*}
which can be made arbitrarily small by fixing $l = l_\gep$,
thanks to Corollary~\ref{th:corcor}, hence we have proven
that $B_2/B_1 \le \gep$ for large $N$.
In a similar fashion, for $B_3$ we can write
\begin{align*}
	& B_3 \;=\; \sum_{n= N/2 + 1}^N \cP_{\gd, T_N}(n \in \tau)
	\,  e^{\phi(\gd, T_N) n} \, P\big( \tau_1^{T_N} > N-n \big) \\
	& \ \;\le\; (const.) \,
	\sum_{n= N/2 + 1}^N \frac{1}{n^{3/2}} \, \frac{1}{\sqrt{N-n+1}}
	\;\le\; \frac{(const.)}{(N/2)^{3/2}} \,
	\sum_{n= N/2 + 1}^N \frac{1}{\sqrt{N-n+1}} \;\le\; \frac{(const.')}{N} \,,
\end{align*}
where we have used the upper bounds in \eqref{eq:bound_renewal} and
\eqref{eq:boundqbis} as well as
the fact that $\phi(\gd, T_N) n = o(1)$ uniformly in $n \le N$,
by \eqref{eq:phineg}. Therefore for large~$N$ we have $B_3 / B_1 \le \gep$
and equation \eqref{eq:priimo} is proven. This implies that, for every $\gep > 0$,
there exists $l_\gep \in \N$ such that for large $N$
\begin{equation} \label{eq:ingredient}
	\bP^{T_N}_{N,\gd}
	\big( \tau^{T_N}_{L_{N,T_N}} < l_\gep \big) \ge 1-\gep\,.
\end{equation}

Next we turn to the proof of the both relations in \eqref{eq:supercrit2}
at the same time.
In view of \eqref{eq:ingredient}, it suffices to show that, for every $\gep > 0$,
we can choose $M \in \N$ and $\eta > 0$ such that for large $N$
\begin{equation} \label{eq:qwerty}
	\bP^{T_N}_{N,\gd} \bigg( \Big\{ \tau^{T_N}_{L_{N,T_N}} < l_\gep \Big\} \cap
	\bigg( \bigg\{ \sup_{n \le N} |S_n| > M \sqrt{N} \bigg\} \cup
	\Big\{ |S_N| \le \eta \sqrt{N} \Big\} \bigg) \bigg)
	\;\le\; \gep \,.
\end{equation}
Summing over the values of $\tau^{T_N}_{L_{N,T_N}}$ and using \eqref{eq:overandover},
the l.h.s. of \eqref{eq:qwerty} is bounded from above by
\begin{equation*}
	\sum_{u=0}^{l_\gep-1} \cP_{\gd, T_N} ( u \in \tau ) \, e^{\phi(\gd, T_N) u}
	\, A_{N,u}(M,\eta) \,,
\end{equation*}
where
\begin{equation*}
	A_{N,u}(M,\eta) \,:=\,
	\frac{P \big( \big\{\tau_1^{T_N} > N-u \big\} \cap
	\big( \big\{ \sup_{n \le N-u} |S_n| > M \sqrt{N} \big\} \cup
	\big\{ |S_{N-u}| \le \eta \sqrt{N} \big\} \big) \big)}
	{Z^{T_N}_{N,\gd}} \,.
\end{equation*}
Therefore equation \eqref{eq:qwerty} will be proven once we show that
we can chose $M, \eta$ such that
$A_{N,u}(M,\eta) \le \gep/l_\gep$, for $N$ large.
For the partition function appearing in the denominator,
applying \eqref{eq:B1new} and \eqref{eq:lbB1new} we easily obtain
$Z^{T_N}_{N,\gd} \ge (const.) / \sqrt{N}$.
Setting $N_u := N-u$ for short,
the numerator in the definition of $A_{N,u}(M,\eta)$ can be bounded from above by
\begin{equation*}
	P\big( |S_i| > 0\,,\, \forall i \le N_u \big) \cdot
	P\bigg( \bigg\{ \sup_{n \le N_u} |S_n| > M \sqrt{N} \bigg\} \cup
	\Big\{ |S_{N_u}| \le \eta \sqrt{N} \Big\}
	\,\bigg|\, |S_i| > 0\,,\, \forall i \le N_u \bigg)\,.
\end{equation*}
It is well-known \cite{cf:Fel1}
that $P\big( |S_i| > 0\,,\, \forall i \le n \big) \le (const.) / \sqrt n$.
Recalling the weak convergence of the random walk conditioned to stay
positive toward the Brownian meander \cite{cf:Bol}, we conclude that for
every fixed $u \le l_\gep$ and for large $N$ we have the bound
\begin{equation}
	A_{N,u}(M,\eta) \;\le\; (const.) \, P\bigg( \bigg\{ \sup_{0 \le t \le 1}
	m_t > M \bigg\} \cup \big\{ m_1 \le \eta \big\} \bigg) \,.
\end{equation}
We can then choose $M$ large and $\eta$ small so as to satisfy the desired bound
$A_{N,u}(M,\eta) \le \gep/l_\gep$, and the proof is completed.\qed


\bigskip

\appendix

\section{On the free energy}


\smallskip
\subsection{Free energy estimates}

\label{sec:fe_estimates}

We determine the asymptotic behavior of $\phi(\gd,T)$
as $T \to \infty$, for fixed $\gd < 0$.
By Theorem~1 in~\cite{cf:CP},
we have $Q_T\big( \phi(\gd,T) \big) = e^{-\gd}$, and furthermore
\begin{equation*}
	Q_T(\gl) \,=\, 1 + \sqrt{e^{-2\gl} - 1} \cdot
    \frac{1 - \cos\big( T \arctan \sqrt{e^{-2\gl} - 1} \big)}{
    \sin \big( T \arctan \sqrt{e^{-2\gl} - 1} \big)} \,,
\end{equation*}
see, e.g., equation (A.5) in \cite{cf:CP}. If we set
\begin{equation} \label{eq:gamma}
    \gamma = \gamma(\gd,T) := \arctan \sqrt{e^{-2\phi(\gd,T)} - 1}\,,
\end{equation}
we can then write
\begin{equation} \label{eq:Qtilde}
    \tilde Q_T \big( \gamma(\gd,T) \big) \;=\; e^{-\gd} \qquad
    \text{where} \qquad
    \tilde Q_T (\gamma) \;=\; 1 + \tan \gamma \cdot
    \frac{1-\cos(T\gamma)}{\sin(T\gamma)}\,.
\end{equation}
Note that $\gamma \mapsto \tilde Q_T (\gamma)$ is an increasing function
with $\tilde Q_T (0) = 1$ and $\tilde Q_T (\gamma) \to + \infty$
as $\gamma \uparrow \frac \pi T$, hence $0 < \gamma(\gd,T)  < \frac\pi T$.
So we have to study the equation $\tilde Q_T (\gamma) = e^{-\gd}$
for $0 < \gamma < \frac \pi T$. An asymptotic development yields
\begin{equation*}
    (1+o(1)) \, \gamma \cdot \frac{1-\cos(T\gamma)}{\sin(T\gamma)}
    \;=\; e^{-\gd}-1\,,
\end{equation*}
where here and in the sequel $o(1)$ is to be understood as $T \to\infty$
with $\gd < 0$ fixed. Setting $x = T \gamma$ gives
\begin{equation*}
    (1+o(1)) \, x \cdot \frac{1-\cos x}{\sin x}     \;=\; T(e^{-\gd}-1)\,,
\end{equation*}
where $0 < x < \pi$. Since the r.h.s. diverges as $T \to\infty$,
$x$ must tend to $\pi$ and a further development yields
\begin{equation*}
    (1+o(1)) \, \frac{2\pi}{\pi - x} = T(e^{-\gd}-1)\,,
\end{equation*}
from which we get $x = \pi - \frac{2\pi}{e^{-\gd}-1} \frac 1T (1+o(1))$
and hence, since $\gamma(\gd,T) = \frac x T$,
\begin{equation} \label{eq:asgamma}
    \gamma(\gd,T) \;=\; \frac{\pi}{T}
    - \frac{2\pi}{e^{-\gd}-1} \, \frac{1}{T^2} (1+o(1))\,.
\end{equation}
Recalling \eqref{eq:gamma}, we have
\begin{equation*}
    \sqrt{e^{-2\phi(\gd,T)} - 1} \;=\;
    \tan \bigg(\frac{\pi}{T}
    - \frac{2\pi}{e^{-\gd}-1} \, \frac{1}{T^2} (1+o(1)) \bigg)\,.
\end{equation*}
Since the function $\gl \mapsto \arctan \sqrt{e^{-2\gl} - 1}$
is decreasing and continuously differentiable, with non-vanishing first derivative, it
follows that
\begin{equation} \label{eq:phineg2}
    \phi(\gd,T) \;=\; - \frac{\pi^2}{2T^2} \bigg( 1 -
    \frac{4}{e^{-\gd} - 1}\,\frac 1T + o\bigg( \frac 1T \bigg) \bigg)\,,
\end{equation}
so that equation \eqref{eq:phineg} is proven.\qed


\smallskip

\subsection{Further estimates}

\label{sec:further_estimates}

We now derive some asymptotic properties of the variables
$(\tau_1, \gep_1)$ under $\cP_{\gd, T}$, as $T\to\infty$
and for fixed $\gd < 0$.

We first focus on $Q^1_T(\phi(\gd,T))$, where
$Q^1_T(\gl) := E(e^{-\gl \tau_1^T} \, \ind_{\{\gep_1^T = 1\}})
= \sum_{n\in\N} e^{-\gl n} q_T^1(n)$. In analogy with the computations above,
by equation (A.5) in~\cite{cf:CP} we can write
\begin{equation*}
	Q^1_T(\phi(\gd,T)) \;=\; \tilde Q^1_T(\gamma(\gd,T))\,, \qquad \quad
	\text{where } \quad \
	\tilde Q^1_T(\gamma) \;:=\; \frac{\tan \gamma}{2\, \sin(T\gamma)}\,,
\end{equation*}
so that from \eqref{eq:asgamma} we obtain as $T \to \infty$
\begin{equation} \label{eq:asQ1}
	Q^1_T(\phi(\gd,T)) \;=\; \frac{\pi}{T} \, \frac{1}
	{2 \cdot \frac{2\pi}{e^{-\gd}-1} \, \frac 1T} \, (1+o(1)) \;\longrightarrow\;
	\frac{e^{-\gd}-1}{4} \,.
\end{equation}
In particular, by \eqref{eq:defPdeltaT} we can write as $T \to \infty$
\begin{equation} \label{eq:asQ1bis}
	\cE_{\gd, T}(\gep_1^2) \;=\; 2 \,\cP_{\gd, T} (\gep_1 = +1) \;=\;
	2 \, e^\gd \, Q^1_T(\phi(\gd,T)) \;\longrightarrow\;
	\frac{1-e^\gd}{2} \,.
\end{equation}

\smallskip

Next we determine the asymptotic behavior of $\cE_{\gd, T}(\tau_1)$
as $T \to\infty$ for fixed $\gd < 0$. Recalling \eqref{eq:taudelta}
we can write
\begin{align} \label{eq:mmean}
    \cE_{\gd, T}(\tau_1) \;&=\; e^\gd\,\sum_{n\in\N} n\, q_T(n) \,
    e^{-\phi(\gd,T) n} \;=\; - e^\gd\cdot Q_T'(\phi(\gd,T))\,,\\
    \cE_{\gd, T}(\tau_1^2) \;&=\; e^\gd\,\sum_{n\in\N} n^2\, q_T(n) \,
    e^{-\phi(\gd,T) n} \;=\;  e^\gd\cdot Q_T''(\phi(\gd,T))\,,
\end{align}
hence the problem is to determine $Q_T'(\gl)$ for $\gl = \phi(\gd,T)$.
Introducing the function $\gamma(\gl) := \arctan \sqrt{e^{-2\gl} - 1}$
and recalling \eqref{eq:Qtilde}, since $Q_T = \tilde Q_T \circ \gamma$ it follows that
\begin{align}\label{eq:dersec}
    Q_T'(\gl) \;&=\; \tilde Q'_T(\gamma(\gl)) \cdot \gamma'(\gl)\,,\\
     Q_T''(\gl) \;&=\; \gamma''(\gl)\cdot \tilde Q'_T(\gamma(\gl))+ (\gamma'(\gl))^2\cdot \tilde Q''_T(\gamma(\gl))\,.
\end{align}
By direct computation
\begin{align}
    \tilde Q'_T(\gamma) \;&=\; \frac{1\cos(T\gamma)}{\sin(T\gamma)}\,\cdot\,
    \bigg(\frac{1}{\cos^2\gamma}\;+\;\frac{T \tan \gamma}{\sin(T\gamma)}\bigg)\;,\\
    \tilde Q''_T(\gamma) \;&=\; \frac{1-\cos(T \gamma)}{\sin(T \gamma)}
    \,\cdot\,
    \bigg(\frac{2 T}{\sin(T \gamma)\, \cos^2 x}\,+\,\frac{2 \sin \gamma}{\cos^3 x}+
    \frac{T^2 \tan \gamma}{\sin^2(T \gamma)}\, (1 - \cos(T\gamma))\bigg)\,,
\end{align}
and
\begin{equation*}
    \gamma'(\gl) \;=\; - \frac{1}{\sqrt{e^{-2\gl} - 1}}\;,\quad \quad \quad \gamma''(\gl) \;=\; - \frac{e^{-2 \gl}}{(e^{-2\gl} - 1)^{3/2}}\;.
\end{equation*}
Recalling \eqref{eq:mmean} and \eqref{eq:gamma}, we have
\begin{equation*}
    \cE_{\gd, T}(\tau_1) \;=\; -e^\gd \cdot
    \tilde Q'_T(\gamma(\gd,T)) \cdot \gamma'(\phi(\gd,T))\,.
\end{equation*}
Now the asymptotic behaviors \eqref{eq:asgamma} and \eqref{eq:phineg2} give
\begin{equation*}
    \tilde Q'_T(\gamma(\gd,T)) \;=\; \frac{e^{-\gd}-1}{\pi} \, T
    \;+\; \frac{(e^{-\gd}-1)^2}{2\pi} \, T^2
    \;+\; o(T)\,, \qquad \
    \gamma'(\phi(\gd,T)) \;=\; - \frac T\pi \;+\; o(T)\,,
\end{equation*}
and
\begin{equation*}
    \tilde Q''_T(\gamma(\gd,T)) \;=\; \frac{(e^{-\gd}-1)^3}{2 \pi^2} \, T^4
    \;+\; o(T^4)\,, \qquad \
    \gamma''(\phi(\gd,T)) \;=\; - \frac {T^3}{\pi^3} \;+\; o(T^3)\,.
\end{equation*}
Combining the preceding relations, we obtain
\begin{align*}
    \cE_{\gd, T}(\tau_1) \;&=\; \frac{e^\gd (e^{-\gd}-1)^2}{2 \pi^2} \, T^3
    \;+\; \frac{1-e^\gd}{\pi^2}\, T^2 \;+\; o(T^2)\,,\\
    \cE_{\gd, T}(\tau_1^2) \;&=\; \frac{e^\gd (e^{-\gd}-1)^3}{2 \pi^4} \, T^6
    \;+\; o(T^6)\,,
\end{align*}
which show that equations \eqref{eq:asET} and \eqref{eq:asET2} hold true.\qed

\medskip

\section{Renewal theory estimates}

This section collects the proofs of Lemma~\ref{th:ineg2} and
Proposition~\ref{th:bound_renewal}.


\smallskip

\subsection{Proof of Lemma~\ref{th:ineg2}}
\label{sec:lemmaineg2}

We recall that, by equation (5.8) in Chapter XIV of~\cite{cf:Fel1}, we have
the following explicit formula for $q_T^j(n)$ (defined in \eqref{eq:defQ}):
\begin{align}\label{eq:equat}
\begin{split}
q^0_T(n) & \;=\; \Bigg(
\frac{2}{T} \sum_{\nu=1}^{\lfloor (T-1)/2\rfloor} \cos^{n-2} \bigg(\frac{\pi \nu}{T}\bigg)
\sin^2 \bigg(\frac{\pi \nu}{T}\bigg) \Bigg) \cdot \ind_{\{n \text{ is even}\}} \,, \\
q^1_T(n) & \;=\; \Bigg(
\frac{1}{T} \sum_{\nu=1}^{\lfloor (T-1)/2\rfloor}
(-1)^{\nu+1} \cos^{n-2} \bigg(\frac{\pi \nu}{T}\bigg)
\sin^2 \bigg(\frac{\pi \nu}{T}\bigg) \Bigg) \cdot \ind_{\{n-T \text{ is even}\}} \,,
\end{split}
\end{align}
hence $q_T(n) = P(\tau_1^T = n) = q_T^0(n) + 2 q_T^1(n)$
is given for $n$ and $T$ even by
\begin{align}\label{eq:q}
q_T(n) \;=\;
\frac{4}{T} \sum_{\nu=1}^{\lfloor (T+2)/4\rfloor}
\cos^{n-2} \bigg(\frac{(2\nu - 1) \pi}{T}\bigg)
\sin^2 \bigg(\frac{(2\nu - 1) \pi}{T}\bigg) \,,
\end{align}
(notice that $\lfloor (T-1)/2 \rfloor = T/2 - 1$ for $T$ even).

We split \eqref{eq:q} in the following way:
we fix $\gep>0$ and we write
\begin{equation} \label{eq:zeroth}
P(\tau_1^T = n) \;=\; V_0(n) \;+\; V_1(n) \;+\; V_2(n)\,,
\end{equation}
where we set
\begin{gather*}
	V_0(n) \;:=\; \frac{4}{T} \cos^{n-2}\left( \frac{\pi}{T} \right)
	\, \sin^2 \left( \frac{\pi}{T} \right) \,, \\
	V_1(n) \;:=\; \frac{4}{T} \,
	\sum_{\nu=2}^{\lfloor \gep T \rfloor}
	\cos^{n-2} \left( \frac{(2\nu - 1) \pi}{T}\right)
	\sin^2\left(\frac{(2\nu - 1) \pi}{T}\right) \,, \\
	V_2(n) \;:=\; \frac{4}{T} \,
	\sum_{\nu= \lfloor \gep T \rfloor + 1 }^{\lfloor (T+2)/4 \rfloor}
	\cos^{n-2} \left(\frac{(2\nu - 1) \pi}{T}\right)
	\sin^2 \left( \frac{(2\nu - 1) \pi}{T} \right) \,.
\end{gather*}
Plainly, as $T \to \infty$ we have
\begin{equation} \label{eq:first}
	V_0(n) \;=\;
	\frac{4 \pi^2}{T^3}\, \left( 1 + o(1) \right) \,
	e^{-g(T) n} \,,
\end{equation}
where $o(\cdot)$ refer as $T \to \infty$, {\sl uniformly
in $n$}. Next we focus on $V_1$:
for $\gep$ small enough and $x\in[0,\pi \gep]$ we have
$\log(\cos(x))\leq -\tfrac{x^2}{3}$, and since $\sin(x)\leq x$ we have
\begin{equation}\label{eq:second}
\begin{split}
	V_1(n) & \;\le\; \frac{4 \pi^2}{T}
	\, \sum_{\nu=2}^{\lfloor \gep T \rfloor}\, \left(\frac{2\nu-1}{T}\right)^2 \,
	e^{-\frac{(n-2) \pi^2}{3} (\frac{2\nu-1}{T})^2} \;\le\;
	(const.) \, \int_{2/T}^{\infty} x^2 \,
	e^{-\frac{\pi^2}{3} n x^2} \, \dd x \\
	& \;=\; \frac{(const.)}{n^{3/2}}
	\int_{2 \sqrt{n}/T}^\infty y^2 \, e^{-\frac{\pi^2}{3} y^2}
	\, \dd y \;\le\; \frac{(const.')}{n^{3/2}}
	e^{-\tfrac{\pi^2 n}{T^2}} \;\le\; \frac{(const.')}{n^{3/2}}
	e^{-g(T) n} \,,
\end{split}
\end{equation}
where the last inequality holds for $T$ large by \eqref{eq:gT}.
The upper bound on $V_2$ is very rough: since $\sin(x) \le x$
and $\cos(x) \le \cos(\pi \gep)$ for $x \in [\pi \gep, \pi/2]$,
we can write
\begin{equation} \label{eq:third}
	V_2(n) \;\le\; \frac{16\pi^2}{T^3} \, \cos^{n-2}(\pi \gep) \,
	\sum_{\nu = \lfloor \gep T \rfloor + 1}^{\lfloor (T+2)/4 \rfloor} \nu^2
	\;\le\; (const.)\, \cos^{n}(\pi\gep) \,.
\end{equation}
Finally, we get a lower bound on $V_1 + V_2$, but only when
$400 \le n \le T^2$.
Since $\log(\cos(x)) \ge -\frac{2}{3} x^2$ and $\sin(x)\ge \frac{x}{2}$
for $x \in [0, \pi/4]$, we can write
\begin{equation}\label{eq:fourth}
\begin{split}
	V_1(n) + & V_2(n) \;\ge\; \frac{\pi^2}{T}
	\, \sum_{\nu=2}^{\lfloor T/8 \rfloor}\, \left(\frac{2\nu-1}{T}\right)^2 \,
	e^{-\frac{2 n \pi^2}{3} (\frac{2\nu-1}{T})^2} \;\ge\; \frac{\pi^2}{2} \,
	\int_{4/T}^{1/4} x^2 \,
	e^{-\frac{2 \pi^2}{3} n x^2} \, \dd x \\
	& \ \ \;=\; \frac{\pi^2}{2n^{3/2}}
	\int_{4 \sqrt{n}/T}^{\sqrt{n}/4} y^2 \, e^{-\frac{\pi^2}{3} y^2}
	\, \dd y \;\ge\; \frac{\pi^2}{2n^{3/2}}
	\int_{4}^{5} y^2 \, e^{-\frac{\pi^2}{3} y^2}
	\, \dd y \;=\; \frac{(const.)}{n^{3/2}} \,.
\end{split}
\end{equation}
Putting together
\eqref{eq:first}, \eqref{eq:second} and \eqref{eq:third},
it is easy to see that the upper bound in \eqref{eq:boundq} holds true
(consider separately the cases $n \le T^2$ and $n > T^2$),
while the lower bound follows analogously from
\eqref{eq:first} and \eqref{eq:fourth}.
To see that also equation \eqref{eq:boundqbis} holds it is
sufficient to sum the bounds in \eqref{eq:boundq} over $n$,
and the proof is completed.\qed


\smallskip

\subsection{Proof of Proposition~\ref{th:bound_renewal}}

\label{sec:bound_renewal}

For convenience, we split the proof in two parts, distinguishing between
the two regimes $n \le T^3$ and $n \ge T^3$.

\smallskip

\subsubsection{The regime $n \le T^3$}

The lower bound in \eqref{eq:bound_renewal} for $n\le T^3$ follows easily from
$\cP_{\gd, T}(n \in \tau) \ge \cP_{\gd, T}(\tau_1 = n)$
together with the lower bound in \eqref{eq:boundren}.
The upper bound requires more work.
We set for $k,n \in \N$
\begin{equation*}
	K_k(n) \;=\; K^T_k(n) \; := \; \cP_{\gd, T}(\tau_k = n) \,,
\end{equation*}
and we note that, by \eqref{eq:boundrenbislb}
and \eqref{eq:gplusphi},
there exists $T_0 > 0$ and $\ga < 1$
such that $\sum_{n=1}^{T^3} K_1(n)  \; \le \; \ga$,
for every $T > T_0$. Since $K_{k+1}(n) = \sum_{m=1}^{n-1}
K_k(m) K_1(n-m)$, an easy induction argument yields
\begin{equation} \label{eq:totalsum}
	\sum_{n=1}^{T^3} K_k(n)  \; \le \; \ga^k \,, \qquad
	\forall k \in \N\,.
\end{equation}
Next we turn to a pointwise upper bound on $K_k(n)$.
From the upper bound in \eqref{eq:boundren},
we know that there exists $C > 0$ such that
$K_1(n) \le C/\min\{n^{3/2}, T^3\}$
for every $n \le T^3$. We now claim that
\begin{equation} \label{eq:induction}
	K_k(n) \;\le\; k^3 \, \ga^{k-1} \,
	\frac{C}{\min\{n^{3/2}, T^3\}} \,, \qquad
	\forall k\in \N \,, \ \forall n \le T^3 \,.
\end{equation}
We argue by induction: we have just observed that this formula holds true
for $k=1$. Assuming now that the formula holds for $k = 1, \ldots, 2m-1$,
we can write for $n \le  T^3$
\begin{equation*}
	K_{2m}(n) \;\le\; 2 \sum_{i=1}^{\lceil n/2 \rceil}
	K_m(i) \, K_m(n-i) \;\le\; 2 \sum_{i=1}^{\lceil n/2 \rceil}
	K_m(i) \left( m^3 \, \ga^{m-1} \,
	\frac{C}{\min\{(n-i)^{3/2}, T^3\}} \right) \,,
\end{equation*}
and since
$\min\{(n-i)^{3/2}, T^3\} \ge \min\{(n/2)^{3/2}, T^3\}
\ge 2^{-3/2} \min\{n^{3/2}, T^3 \}$
for $i$ in the range of summation, from \eqref{eq:totalsum} we get
\begin{equation*}
\begin{split}
	K_{2m}(n) \;\le\; 2^{5/2} \, m^3 \, \ga^{m-1} \,
	\frac{C}{\min\{n^{3/2}, T^3 \}} \sum_{i=1}^{\lceil n/2 \rceil}
	K_m(i) \;\le\;  (2m)^3 \, \ga^{2m-1} \,
	\frac{C}{\min\{n^{3/2}, T^3 \}} \,,
\end{split}
\end{equation*}
so that \eqref{eq:induction} is proven
(we have only checked it when $k=2m$, but the
case $k=2m+1$ is completely analogous).
For $n \le T^3$ we can then write
\begin{equation*}
	\cP_{\gd, T}(n\in\tau) \;=\; \sum_{k=0}^\infty K_k(n) \;\le\;
	\frac{C}{\min\{n^{3/2}, T^3 \}} \, \sum_{k=0}^{\infty} k^3 \, \ga^{k-1}
	\;=\; \frac{(const.)}{\min\{n^{3/2}, T^3 \}}\,,
\end{equation*}
hence the upper bound in \eqref{eq:bound_renewal} is proven.
\qed


\smallskip

\subsubsection{The regime $n \ge T^3$}

We start proving the lower bound in \eqref{eq:bound_renewal}
for $n\ge T^3$.
Setting $\gamma_m := \inf\{ k\ge m:\, k \in \tau\}$
we can write
\begin{align*}
	& \cP_{\gd, T}(n \in \tau) \;\ge\;
	\cP_{\gd, T} \big( \tau \cap [n-T^3, n-1] \ne \emptyset,\, n \in \tau \big) \\
	& \;=\; \sum_{k=n - T^3}^{n-1} \cP_{\gd, T} \big( \mu_{n-T^3} = k \big)
	\, \cP_{\gd, T} \big(n-k \in \tau \big) \;\ge\;
	\frac{(const.)}{T^3} \, \cP_{\gd, T}
	\big( \tau \cap [n-T^3, n-1] \ne \emptyset \big) \,,
\end{align*}
where we have applied the lower bound in \eqref{eq:bound_renewal}
to $\cP_{\gd, T} \big(n-k \in \tau \big)$, because $n-k \le T^3$.
It then suffices to show that there exist $c,T_0 > 0$ such that
for $T > T_0$ and $n \ge T^3$
\begin{equation*}
	\cP_{\gd, T} \big( \tau \cap [n-T^3, n-1] \ne \emptyset \big) \;>\; c\,.
\end{equation*}
We are going to prove the equivalent statement
\begin{equation} \label{eq:comparison}
	\cP_{\gd, T} \big( \tau \cap [n-T^3, n-1] \ne \emptyset \big) \;\ge\;
	C \, \cP_{\gd, T} \big( \tau \cap [n-T^3, n-1] = \emptyset \big) \,,
\end{equation}
for a suitable $C > 0$. We have
\begin{equation} \label{eq:nonempty}
\begin{split}
	& \cP_{\gd, T} \big( \tau \cap [n-T^3, n-1] \ne \emptyset \big) \;=\;
	\sum_{\ell = 0}^{n-T^3 -1}
	\cP_{\gd, T} \big( \ell \in \tau \big) \, \sum_{k = n - T^3}^{n-1}
	\cP_{\gd, T} (\tau_1 = k - \ell) \\
	& \ge \; (const.) \, \sum_{\ell = 0}^{n-T^3 -1}
	\cP_{\gd, T} \big( \ell \in \tau \big) \,
	\left( e^{-(\phi(\gd,T)+g(T))(n-T^3-\ell)} - e^{-(\phi(\gd,T)+g(T))(n-\ell)}
	\right) \,,
\end{split}
\end{equation}
having applied \eqref{eq:boundrenbislb}. Analogously, applying
\eqref{eq:boundrenbisub} we get
\begin{equation} \label{eq:empty}
\begin{split}
	& \cP_{\gd, T} \big( \tau \cap [n-T^3, n-1] = \emptyset \big) \;=\;
	\sum_{\ell = 0}^{n-T^3 -1}
	\cP_{\gd, T} \big( \ell \in \tau \big) \, \sum_{k = n}^{\infty}
	\cP_{\gd, T} (\tau_1 = k - \ell) \\
	& \quad \; \le \; (const.) \, \sum_{\ell = 0}^{n-T^3 -1}
	\cP_{\gd, T} \big( \ell \in \tau \big) \,
	e^{-(\phi(\gd,T)+g(T))(n-\ell)} \,,
\end{split}
\end{equation}
having used the upper bound in \eqref{eq:boundq}.
However we have
\begin{equation*}
	\frac{e^{-(\phi(\gd,T)+g(T))(n-T^3-\ell)} - e^{-(\phi(\gd,T)+g(T))(n-\ell)}}
	{e^{-(\phi(\gd,T)+g(T))(n-\ell)}} \;=\;
	e^{T^3(\phi(\gd,T)+g(T))} -1 \;\xrightarrow{T\to\infty}\;
	e^{\frac{2 \pi^2}{(e^{-\gd}-1)}} - 1\,,
\end{equation*}
thanks to \eqref{eq:gplusphi}, so that \eqref{eq:comparison}
is proven.

It remains to prove the upper bound in \eqref{eq:bound_renewal}
for $n \ge T^3$. Notice first that
\begin{align*}
	& \cP_{\gd, T} \big(\tau \cap [n-T^3, n-T^2] \ne \emptyset \,,\,
	n \in \tau \big)
	\;=\; \sum_{k=n-T^3}^{n-T^2} \cP_{\gd, T} (\gamma_{n-T^3} = k)
	\, \cP_{\gd, T} (n-k \in \tau) \\
	& \quad \;\le\; \frac{(const.)}{T^3} \,
	\cP_{\gd, T} \big(\tau \cap [n-T^3, n-T^2] \ne \emptyset \big) \;\le\;
	\frac{(const.)}{T^3} \,,
\end{align*}
having applied the upper bound in \eqref{eq:bound_renewal}
to $\cP_{\gd, T} (n-k \in \tau)$, because
$T^2 \le n - k \le T^3$. If we now show that
there exist $c,T_0 > 0$ such that
for $T > T_0$ and for $n > T^3$
\begin{equation} \label{eq:babao}
	\cP_{\gd, T} \big( \tau \cap [n-T^3, n-T^2] \ne \emptyset
	\,\big|\, n \in \tau \big) \;\ge\; c\,,
\end{equation}
it will follow that
\begin{equation*}
	\cP_{\gd, T} ( n \in \tau ) \;\le\; \frac{1}{c} \,
	\cP_{\gd, T} \big( \tau \cap [n-T^3, n-T^2] \ne \emptyset\,,\,
	n \in \tau \big) \;\le\; \frac{(const.')}{T^3}\,,
\end{equation*}
and we are done.
Instead of \eqref{eq:babao}, we prove the equivalent relation
\begin{equation} \label{eq:babao2}
	\cP_{\gd, T} \big( \tau \cap [n-T^3, n-T^2] \ne \emptyset
	,\, n \in \tau \big) \;\ge\; C
	\cP_{\gd, T} \big( \tau \cap [n-T^3, n-T^2] = \emptyset
	,\, n \in \tau \big)\,,
\end{equation}
for some $C>0$. We start considering the l.h.s.:
\begin{equation} \label{eq:wehave}
\begin{split}
	& \cP_{\gd, T} \big( \tau \cap [n-T^3, n-T^2] \ne \emptyset
	,\, n \in \tau \big) \\
	& \;=\; \sum_{m=0}^{n-T^3 -1}
	\cP_{\gd, T} (m \in \tau) \, \sum_{\ell = n-T^3}^{n-T^2}
	\cP_{\gd, T} (\tau_1 = \ell-m) \, \cP_{\gd, T}
	(n-\ell \in \tau) \,.
\end{split}
\end{equation}
Notice that $\cP_{\gd, T} (n-\ell \in \tau) \ge (const.)/T^3$
for $n-\ell \in 2\N$ by the lower bound in \eqref{eq:bound_renewal}. Equation
\eqref{eq:boundrenbislb} then yields
\begin{equation*}
\begin{split}
	& \sum_{\ell = n-T^3}^{n-T^2} 	\cP_{\gd, T} (\tau_1 = \ell-m)
	\;\ge\; (const.) \,
	\left( e^{-(\phi(\gd,T)+g(T))(n-T^3-m)} - e^{-(\phi(\gd,T)+g(T))(n-T^2-m)}
	\right)\\
	& \;=\; (const.) \, e^{-(\phi(\gd,T)+g(T))(n-T^3-m)} \,
	(1- e^{-(\phi(\gd,T)+g(T))(T^3-T^2)}) \\
	& \;\ge\; (const.') \, e^{-(\phi(\gd,T)+g(T))(n-T^3-m)}
	\;\ge\; (const.'') \, e^{-(\phi(\gd,T)+g(T))(n-m)}\,,
\end{split}
\end{equation*}
having used repeatedly \eqref{eq:gplusphi}. Coming back
to \eqref{eq:wehave}, we obtain
\begin{equation} \label{eq:qquasi}
\begin{split}
	& \cP_{\gd, T} \big( \tau \cap [n-T^3, n-T^2] \ne \emptyset
	,\, n \in \tau \big) \\
	& \quad \;\ge\; \frac{(const.)}{T^3} \,\sum_{m=0}^{n-T^3 -1}
	\cP_{\gd, T} (m \in \tau) \, e^{-(\phi(\gd,T)+g(T))(n-m)} \,.
\end{split}
\end{equation}
Next we focus on the r.h.s. of \eqref{eq:babao2}:
\begin{equation} \label{eq:wehave2}
\begin{split}
	& \cP_{\gd, T} \big( \tau \cap [n-T^3, n-T^2] = \emptyset
	,\, n \in \tau \big) \\
	& \;=\; \sum_{m=0}^{n-T^3 -1}
	\cP_{\gd, T} (m \in \tau) \, \sum_{\ell = n-T^2}^{n}
	\cP_{\gd, T} (\tau_1 = \ell-m) \, \cP_{\gd, T}
	(n-\ell \in \tau) \,.
\end{split}
\end{equation}
Since $\ell - m \ge T^3 - T^2$, from the upper bound
in \eqref{eq:boundren} we get
\begin{equation*}
	\cP_{\gd, T} (\tau_1 = \ell-m) \;\le\; \frac{(const.)}{T^3}
	\, e^{-(\phi(\gd,T) + g(T))(\ell - m)} \;\le\; \frac{(const.')}{T^3}
	\, e^{-(\phi(\gd,T) + g(T))(n - m)} \,,
\end{equation*}
because $n - \ell \le T^2$ (recall \eqref{eq:gplusphi}).
Furthermore, by the upper bound
in \eqref{eq:bound_renewal} applied to $\cP_{\gd, T} (n-\ell \in \tau)$,
for $n - \ell \le T^2$, we obtain
\begin{equation*}
	\sum_{\ell = n-T^2}^{n} \cP_{\gd, T} (n-\ell \in \tau)
	\;\le\; (const.) \sum_{\ell = n-T^2}^{n} \, \frac{1}{(n-\ell)^{3/2}}
	\;\le\; (const.')\,,
\end{equation*}
and coming back to \eqref{eq:wehave2} we get
\begin{equation} \label{eq:qquasi2}
\begin{split}
	& \cP_{\gd, T} \big( \tau \cap [n-T^3, n-T^2] = \emptyset
	,\, n \in \tau \big) \\
	& \;\le\; \frac{(const.')}{T^3} \sum_{m=0}^{n-T^3 -1}
	\cP_{\gd, T} (m \in \tau) \, e^{-(\phi(\gd,T) + g(T))(n - m)}\,.
\end{split}
\end{equation}
Comparing \eqref{eq:qquasi} and \eqref{eq:qquasi2} we see that
\eqref{eq:babao2} is proven and this completes the proof.
\qed



\bigskip

\end{document}